\documentclass[11pt]{amsart}
\usepackage{verbatim,amssymb,amsmath,graphicx,hyperref}

  \scrollmode

\newcommand{\R}{{\mathbb R}}
\newcommand{\N}{{\mathbb N}}

\newcommand{\EE}{{\mathbb E}}
\newcommand{\PP}{{\mathbb P}}

\newcommand{\eul}{{\widehat X}}
\newcommand{\Tm}{{\mathcal T}}
\newcommand{\eultr}{{\widehat Z}}
\newcommand{\ind}{1}
\newcommand{\usn}{\underline {s}_n}
\newcommand{\utn}{\underline {t}_n}

\newcommand{\upsn}{\overline {s}^n}
\newcommand{\uptn}{\overline {t}^n}

\newcommand{\sgn}{\operatorname{sgn}}
\newcommand{\eps}{\varepsilon}

\newcommand{\F}{{\mathcal F}}

\theoremstyle{plain}
\newtheorem{theorem}{Theorem}
\newtheorem{prop}{Proposition}
\newtheorem{lemma}{Lemma}

\theoremstyle{definition}
\newtheorem{rem}{Remark}

\usepackage[textsize=footnotesize]{todonotes}

\begin{document}
\title[]{Existence, uniqueness and approximation of solutions of SDEs with superlinear coefficients in the presence of discontinuities of the drift coefficient}

\author[M\"uller-Gronbach]
{Thomas M\"uller-Gronbach}
\address{
Faculty of Computer Science and Mathematics\\
University of Passau\\
Innstrasse 33 \\
94032 Passau\\
Germany} \email{thomas.mueller-gronbach@uni-passau.de}

\author[Sabanis]
{Sotirios Sabanis}
\address{
School of Mathematics\\
University of Edinburgh\\
James Clerk Maxwell Building \\
Peter Guthrie Tait Road\\
Edinburgh\\
EH9 3FD\\
Scotland} \email{S.Sabanis@ed.ac.uk }

\author[Yaroslavtseva]
{Larisa Yaroslavtseva}
\address{
Faculty of Computer Science and Mathematics\\
University of Passau\\
Innstrasse 33 \\
94032 Passau\\
Germany} \email{larisa.yaroslavtseva@uni-passau.de}

\begin{abstract}
Existence, uniqueness, and $L_p$-approximation results are presented for
scalar stochastic differential equations (SDEs)   by considering the case where, the drift coefficient has finitely many spatial discontinuities while both coefficients can grow superlinearly (in the space variable).
These discontinuities are described by a piecewise local Lipschitz continuity and a piecewise monotone-type condition   while the diffusion coefficient is assumed to be locally Lipschitz continuous and  non-degenerate at the discontinuity points of the drift coefficient. Moreover, the superlinear nature of the coefficients is dictated by a suitable coercivity condition and a polynomial growth of the  (local)  Lipschitz constants of the coefficients. Existence and uniqueness of strong solutions of such SDEs  are obtained. Furthermore, the classical $L_p$-error rate $1/2$, for  a   suitable range of   values of $p$, is recovered for a tamed Euler scheme which is used for approximating these solutions.
To the best of  the authors'  knowledge, these are the first existence, uniqueness and approximation results for  this class of   SDEs.
\end{abstract}
\maketitle

\section{Introduction}
Let
$ ( \Omega, \mathcal{F}, \PP ) $
be a probability space with a normal filtration
$ ( \mathcal{F}_t )_{ t \in [0,\infty) } $ and consider a scalar autonomous stochastic differential equation (SDE)
\begin{equation}\label{sde0}
\begin{aligned}
dX_t & = \mu(X_t) \, dt + \sigma(X_t) \, dW_t, \quad t\in [0,\infty),\\
X_0 & = x_0,
\end{aligned}
\end{equation}
where $x_0\in\R$, $\mu, \sigma\colon\R\to\R$ are measurable functions and $W$ is a $1$-dimensional Brownian motion with respect to $ ( \mathcal{F}_t )_{ t \in [0,\infty) } $.

It is well-known that if the coefficients $\mu$ and $\sigma$  are globally Lipschitz continuous then the SDE \eqref{sde0} admits a unique strong solution $X$   which
can be approximated  by the  Euler scheme with an $L_p$-error rate $1/2$, for all $p\in[1, \infty)$,  at any given time $T>0$. For brevity, we consider $T=1$ henceforth.

For a  classical existence and uniqueness result for SDEs with superlinearly growing (but continuous) coefficients  see, e.g., \cite[Theorem 3.1.1]{PrevotRoeckner2007}.  It guarantees the existence of a unique strong solution $X$ of  \eqref{sde0}  if  $\mu$ and $\sigma$ are locally Lipschitz continuous and satisfy   the weak
coercivity condition. Note that the classical Euler scheme is known to diverge  in the $L_1$-sense for many
SDEs of this kind,
see \cite{hjk11}.
As a consequence, there has been a steadily increasing
body of research on new methods of approximation and corresponding $L_p$-error rates for such SDEs over the past decade, see, e.g., \cite{Beynetal2014, Beynetal2016,   FangGiles,  HutzenthalerJentzen2014Memoires, HutzenthalerJentzen2014, hjk12, KumarSabanis2016, Mao2016, Sabanis2013ECP, Sabanis2016,  TretyakovZhang2013, WangGan2013}. In particular, in \cite{Beynetal2014, FangGiles,  hjk12,  Mao2016, Sabanis2013ECP,  Sabanis2016, TretyakovZhang2013} an $L_p$-error rate of at least $1/2$ has been proven for approximating $X_1$ by explicit Euler-type methods,
e.g.,  tamed, projected or  truncated Euler schemes,  for suitable ranges of the values of $p$
and
for  subclasses of such SDEs with coefficients that at least satisfy  a monotone-type condition and  a coercivity condition
and are locally Lipschitz continuous with a polynomially growing  (local)  Lipschitz constant. We add that important applications of these results are emerging in areas of intense interest, due to their central role in Data Science and AI, such as MCMC sampling algorithms, see \cite{tula, hola:19}, and stochastic optimizers for fine tuning (artificial) neural networks and, more broadly, for solving non-convex stochastic optimization problems, see \cite{tusla, Poula}.

For a classical existence and uniqueness result for SDEs with a discontinuous drift coefficient see \cite{ MR0336813}.  Under the assumption that the  diffusion coefficient $\sigma$ is  bounded, bounded away from zero (thus nowhere degenerate)
 and globally Lipschitz continuous
  the latter paper
 provides the existence of a unique strong solution $X$ of  \eqref{sde0}  even if $\mu$ is only measurable and bounded.   Recently, in~\cite{LS16} an  existence and uniqueness result  for  SDEs  with a discontinous drift coefficient  has been proven under much weaker assumptions on the diffusion coefficient. This result states that the SDE \eqref{sde0} admits a unique strong solution $X$ if the drift coefficient $\mu$ has  finitely many discontinuity points and is piecewise Lipschitz continuous   and the diffusion coefficient $\sigma$ is globally Lipschitz continuous and non-degenerate at the  discontinuity points of $\mu$.

 The subject of $L_p$-approximation of solutions of SDEs with a discontinuous drift coefficient has been intensively studied in recent years, see ~\cite{DG18, dareiotis2021, GLN17,  LS16,  LS15b, LS18,  MGY20, MGY19b, NS19, NSS19, Tag16, Tag2017b, Tag2017a, Y21}. In particular, in \cite{DG18, dareiotis2021, LS18, MGY20, NSS19, Tag16, Tag2017b} positive $L_p$-error rates for approximating $X_1$ by explicit Euler-type methods have been proven for such SDEs.
 Under the
  above mentioned
 existence and uniqueness assumptions on $\mu$ and $\sigma$ from~\cite{LS16}, an $L_p$-error rate of at least $1/2$ for the Euler scheme for all $p\in[1, \infty)$ has been recovered in \cite{ MGY20} and an $L_2$-error rate of at least $1/2-$ for an adaptive Euler scheme has been shown in  \cite{NSS19}. Furthermore, in \cite{ dareiotis2021} an $L_p$-error rate of at least $1/2$  for the Euler scheme for all $p\in[1, \infty)$ has been proven in the case when $\mu$ is  measurable and bounded and $\sigma$ is $C^2_b$ and bounded away from zero.
 We add that lower error bounds that hold for any approximation based on finitely many evaluations of the driving Brownian motion are established in~\cite{hhmg2019,  MGY21}.

Existence, uniqueness and approximation of a strong solution of  \eqref{sde0} in the case of superlinearly growing  coefficients $\mu$ and $\sigma$ in the presence of discontinuities of the drift coefficient $\mu$
has, to the best of  the authors'  knowledge, not been studied
in the literature  previously. This article
closes   this gap by allowing both discontinuity and superlinear growth to coexist  as properties  of the drift coefficient.

To be more precise,   it is assumed   that the drift  coefficient $\mu$ has  finitely many discontinuity points and is  piecewise locally Lipschitz continuous while the diffusion coefficient $\sigma$ is  locally Lipschitz continuous   and  non-degenerate at the discontinuity points of $\mu$.  Moreover, $\mu$ and $\sigma$ satisfy a piecewise monotone-type condition and a coercivity condition and the Lipschitz constants of both $\mu$ and $\sigma$
satisfy a polynomial growth condition.

It is proved  that
  the SDE \eqref{sde0} admits a unique strong solution $X$ under these assumptions,
 see Theorem \ref{Thm0},
  and that
 $X_1$ can be approximated by a tamed Euler scheme with an $L_p$-error rate of at least $1/2$ for suitable ranges of the values of $p$.
 The latter result
 is an immediate consequence of  Theorem \ref{Thm1}, which states that  the maximum error of the time-continuous tamed Euler scheme on the time interval $[0,1]$
achieves at least the rate
$1/2$  in
the
$p$-th mean sense.
The proof of Theorem~\ref{Thm1} is based on a rigorous analysis of the
$p$-th mean of the total amount of times
$t\in[0,1]$, for which the  time-continuous tamed Euler scheme at time $t$ and its value at the closest grid point to the left from $t$ lie on different sides of a discontinuity point of the drift coefficient $\mu$.

 Moreover,  the piecewise linear interpolation of the tamed Euler scheme is considered
and its performance
 globally on the time interval $[0,1]$  is examined.  Using Theorem~\ref{Thm1}  yields that
the pathwise $L_q$-error of the piecewise linear interpolated  tamed Euler scheme
achieves at least the rate $1/2$,  if $q<\infty$, and at least the rate $1/2$, up to a log factor,  if $q=\infty$ in the $p$-th mean sense for suitable ranges of the values of $p$, see
 Theorem~\ref{Thm2}.

 In a similar direction but, independently of this work,
  existence and uniqueness of a strong solution of \eqref{sde0} as well as an $L_2$-error rate of at least $1/2$ for approximating $X_1$ by a tamed Euler scheme are presented in  \cite{SS22}. These refer, however, to  the case of a discontinuous and superlinearly growing drift coefficient $\mu$  when the diffusion coefficient $\sigma$ is assumed to be globally Lipschitz continuous.

Finally, although only scalar SDEs are considered in this work, it can be argued that its proof techniques can be naturally   extended to cover an appropriate multidimensional setting as well. Nevertheless, the proof of
such a result is to be the subject of future work.

A brief description of the content of the paper follows. The  assumptions on the coefficients $\mu$ and $\sigma$, the existence and uniqueness result,  Theorem \ref{Thm0}, and  the error estimates, Theorem~\ref{Thm1} and Theorem~\ref{Thm2}, are stated in Section~\ref{s3}. Section~\ref{s4} contains the
proofs of these theorems.

\section{Main results}\label{s3}

 Let $p_0, p_1\in[2, \infty)$, $\ell_{\mu}\in (0,\infty)$ and $\ell_{\sigma}\in[0,\ell{_\mu}/2]$.
  It is assumed henceforth   that the coefficients $\mu\colon\R\to\R$ and $\sigma\colon\R\to\R$ of the SDE \eqref{sde0} satisfy the following  conditions.
 \begin{itemize}
 \item[(A1)] There exists $c\in(0, \infty)$  such that for all $x\in\R$,
 \[
 2x \cdot \mu(x)+(p_0-1)\cdot \sigma^2(x)\leq c\cdot (1+x^2).
 \]
\item[(A2)] There exist $c\in (0, \infty)$, $k\in\N$ and $\xi_0, \ldots, \xi_{k+1}\in [-\infty,\infty]$ with $-\infty=\xi_0<\xi_1<\ldots < \xi_k <\xi_{k+1}=\infty$ such that
 for all $i\in\{1, \ldots, k+1\}$ and all $x,y\in (\xi_{i-1}, \xi_i)$,\\[-.1cm]
 \begin{itemize}
 \item[(i)]
 $2(x-y)\cdot (\mu(x)-\mu(y))+(p_1-1)\cdot (\sigma(x)-\sigma(y))^2\leq c\cdot |x-y|^2$, and\\[-.2cm]
\item[(ii)]
$|\mu(x)-\mu(y)|\leq c\cdot (1+|x|^{\ell_\mu}+|y|^{\ell_\mu})\cdot |x-y|$.\\[-.1cm]
\end{itemize}
\item[(A3)]  There exists $c\in (0, \infty)$ such that
 for all $x,y\in\R$,
 \[
|\sigma(x)-\sigma(y)|\leq c\cdot (1+|x|^{\ell_{\sigma}}+|y|^{\ell_{\sigma}})\cdot |x-y|.
\]
\item[(A4)]  $\sigma(\xi_i) \neq 0$ for all $i\in\{1,\ldots,k\}$.
\end{itemize}

\begin{rem}\label{rem1}
Note that (A3) implies that  $\sigma$ is continuous and there exists $c\in (0, \infty)$ such that  for all $x\in\R$,
\begin{equation}\label{tmg2}
|\sigma(x)|\leq c\cdot (1+|x|^{\ell_\sigma+1}).
\end{equation}
Furthermore, it is easy to check that (A2)(ii) implies that there exists $c\in (0, \infty)$ such that for all $x\in\R$,
\begin{equation}\label{tmg3}
|\mu(x)|\leq c\cdot (1+|x|^{\ell_\mu+1}).
\end{equation}
\end{rem}

We start with the existence and uniqueness result.

\begin{theorem}\label{Thm0}
Assume (A1) to (A4). Then the SDE \eqref{sde0} has a unique strong solution $X$.
\end{theorem}

We turn to the problem of approximating $X_1$ as well as $(X_t)_{t\in[0,1]}$.

For $n\in\N$ we define a  time-continuous tamed Euler scheme $\eul_{n}=(\eul_{n,t})_{t\in[0,1]}$   on $[0,1]$ with step-size $1/n$ 
by
$\eul_{n,0}=x_0$ and
\begin{equation}\label{te}
\eul_{n,t}=\eul_{n,i/n}+\mu_n(\eul_{n,i/n}) \cdot (t-i/n)+\sigma_n(\eul_{n,i/n})\cdot (W_t-W_{i/n})
\end{equation}
for $t\in (i/n,(i+1)/n]$
and $i\in\{0,\ldots,n-1\}$,
where
\begin{equation}\label{tamedcoeff}
\mu_n(x) = \frac{\mu(x)}{1+n^{-1/2}|x|^{\ell_\mu}} \quad \mbox{ and }\quad  \sigma_n(x) = \frac{\sigma(x)}{1+n^{-1/2}|x|^{\ell_\mu}}
\end{equation}
for all $x\in\R$. We have the following error estimates for $\eul_{n}$.

\begin{theorem}\label{Thm1}\mbox{}
			 Let $\mu$ and $\sigma$ satisfy (A1) to (A4) with $p_0 > 2(\ell_\mu +  \max(\ell_\mu, 2\ell_\sigma +2)+ 1)$ and $p_1>2$.
		Then, for every $p\in (0,p_1)\cap (0,\frac{p_0}{\ell_\mu +  \max(\ell_\mu, 2\ell_\sigma +2)+ 1})$,
		there exists $c\in (0,\infty)$ such that, for all $n\in\N$,
		\begin{equation}\label{result4a}
			\EE[ \|X-\eul_{n}    \|_\infty^p]^{1/p}	\le  c/\sqrt{n}.	
		\end{equation}

\end{theorem}

Next, we study the performance of the piecewise linear interpolation
$\overline{X}_n = (\overline X_{n,t})_{t\in [0,1]}$
of the time-discrete tamed Euler scheme, i.e.,
\[
\overline X_{n,t} = (n\cdot t-i)\cdot\widehat X_{n,(i+1)/n} + (i+1-n\cdot t)\cdot \widehat X_{n,i/n}
\]
for $t\in [i/n,(i+1)/n]$
 and $i\in\{0,\ldots,n-1\}$. We have the following error estimates for $\overline X_n$.

\begin{theorem}\label{Thm2}\mbox{}
		 Let $\mu$ and $\sigma$ satisfy (A1) to (A4) with $p_0 > 2(\ell_\mu +  \max(\ell_\mu, 2\ell_\sigma +2)+ 1)$ and $p_1>2$.
		Then, for every $p\in (0,p_1)\cap (0,\frac{p_0}{\ell_\mu +  \max(\ell_\mu, 2\ell_\sigma +2)+ 1})$ and for every $q\in [1,\infty]$,
		there exists $c\in (0,\infty)$ such that, for all $n\in\N$,
		\begin{equation}\label{result4}
			\EE[ \|X-\overline X_n    \|_q^p]^{1/p}	\le  \begin{cases} c/\sqrt{n}, &\text{ if }q\in [1,\infty), \\
				c \sqrt{\ln (n+1)}/\sqrt{n} , &\text{ if }q=\infty.
			\end{cases}
		\end{equation}	
\end{theorem}

\begin{rem}
 For technical reasons
 we have excluded
  the case $\ell_\mu=0$ in our setting. If $\ell_\mu=0$
 then $\ell_\sigma=0$  and therefore
 $\mu$ is piecewise Lipschitz continuous and $\sigma$ is Lipschitz continuous and non-degenerate at the discontinuity points of $\mu$. As already mentioned in the introduction, under the latter assumptions,  existence and uniquness of a strong solution of \eqref{sde0} has been shown in \cite{LS16}. Moreover, in \cite{MGY20} the estimates \eqref{result4a} and \eqref{result4} have been proven for all $p\in[0, \infty)$ for $\widehat X_n$ and $\overline X_n$ being the time-continuous Euler scheme and the piecewise linear interpolation of the time-discrete  Euler scheme, respectively.
 \end{rem}

\begin{rem}
	If the drift coefficient $\mu$ is  continuous then the conditions (A2)(i) and (A2)(ii)  hold globally for all $x,y\in \R$.
	As already mentioned in the introduction the existence and uniqueness of a strong solution of \eqref{sde0} is well-known in this case, see e.g. \cite[Theorem 3.1.1]{PrevotRoeckner2007}. Moreover, in \cite{Sabanis2016} the estimate \eqref{result4a}  has been proven for all  $p\in(0, p_1)\cap (0,\frac{p_0}{2\ell_\mu +1})$ under the assumption that $p_0\geq 4\ell_\mu+2$,
	however, see Remark~\ref{gap}.
\end{rem}

\section{Proofs}\label{Proofs}\label{s4}
We proceed with the proof of the main results. We define
\[
\utn = \lfloor n\cdot t\rfloor / n
\]
for every $n\in\N$ and every $t\in [0,1]$.

We briefly describe the
structure
of this section.
In Subsection~\ref{4.3} we introduce a transformation, which is used to switch, by applying It\^{o}'s formula, from the SDE ~\eqref{sde0} to an SDE with superlinearly growing but continuous coefficients, we provide  crucial properties of this transformation and we prove Theorem~\ref{Thm0}. In Subsection~\ref{4.1} we provide $L_p$-estimates
of the solution $X$ and the time-continuous tamed Euler scheme $\eul_n$. Subsection~\ref{4.2} containes  occupation time estimates for  $\eul_n$, which finally lead to the  $p$-th mean estimate
\[
\EE\Bigl[\Bigl|\int_0^1  \ind_{\{(\eul_{n,t}-\xi_i)\cdot(\eul_{n, \utn}-\xi_i)\leq 0\}}\,dt\Bigr|^p\Bigr]^{1/p}\leq c\,n^{-1/2},
\]
of the
Lebesgue measure of the set
of times $t$ of a sign change of $\eul_{n,t}-\xi_i$ relative to the sign of $\eul_{n,\utn}-\xi_i$
for every $i=1, \ldots, k$, see Proposition~\ref{prop1}. The latter result is a crucial tool for the error analysis of the tamed Euler scheme $\eul_n$. Using the results of Subsections~\ref{4.3}, ~\ref{4.1} and~\ref{4.2} we then derive the error estimates in Theorem~\ref{Thm1} and Theorem ~\ref{Thm2} in Subsections~\ref{proof1} and  ~\ref{4.5}, respectively.

\subsection{The transformation}\label{4.3}

In this subsection we introduce a transformation $G\colon\R\to\R$, see \eqref{G} below, which allows us to switch, by applying the  It\^{o}'s formula, from the SDE~\eqref{sde0} to an SDE with coefficients satisfying  (A1), (A3) and
\begin{itemize}
\item[(A2$'$)] There exists $c\in (0,\infty)$ such that
 for all $x,y\in \R$,
 \begin{itemize}
 \item[(i)]
 $2(x-y)\cdot (\mu(x)-\mu(y))+(p_1-1)\cdot (\sigma(x)-\sigma(y))^2\leq c\cdot |x-y|^2$, and\\[-.2cm]
\item[(ii)]
$|\mu(x)-\mu(y)|\leq c\cdot (1+|x|^{\ell_\mu}+|y|^{\ell_\mu})\cdot |x-y|$.\\[-.1cm]
\end{itemize}
\end{itemize}

To this end we proceed as follows. For all $k\in\N$,
\[
z\in\Tm_k=\{(z_1,\dots,z_k)\in\R^k\colon z_1<\dots<z_k\}
\]
 and $\alpha=(\alpha_1,\dots,\alpha_k)\in\R^k$ we put
\[
\rho_{z,\alpha} =  \begin{cases}
\frac{1}{8 |\alpha_1|}, & \text{if }k=1, \\
\min\bigl(\bigl\{\frac{1}{8 |\alpha_i|}\colon i\in \{1, \ldots, k\}\bigr\} \cup \bigl\{ \frac{z_i-z_{i-1}}{2}\colon i\in \{2, \ldots, k\}\bigr\}\bigr),& \text{if }k\geq 2,
\end{cases}
\]
where we use the convention $1/0 =\infty$. Let $\phi\colon\R\to\R$ be given by
\begin{equation}\label{phi}
\phi(x)=(1-x^2)^4\cdot \ind_{[-1, 1]}(x).
\end{equation}
For all $k\in\N$, $z\in \Tm_k$, $\alpha\in\R^k$ and $\nu\in (0,\rho_{z,\alpha})$, we define a function $G_{z,\alpha,\nu}\colon\R\to\R$ by
\begin{equation}\label{fct1}
G_{z,\alpha,\nu}(x) = x+\sum_{i=1}^k \alpha_i\cdot (x-z_i)\cdot |x-z_i|\cdot \phi \Bigl(\frac{x-z_i}{\nu}\Bigr).
\end{equation}

The following lemma
is known from~\cite{MGY19b}.
It provides the properties of the functions $G_{z,\alpha,\nu}$ that are crucial for our purposes.

\begin{lemma}\label{lemx1}
Let $k\in\N$, $z\in \Tm_k$, $\alpha\in\R^k$, $\nu\in (0,\rho_{z,\alpha})$ and set $z_0=-\infty$ and $z_{k+1}= \infty$. The function $G_{z,\alpha,\nu}$ has the following properties.
\begin{itemize}
\item[(i)] $G_{z,\alpha,\nu}$ is differentiable with
\[
0<\inf_{x\in\R} G_{z,\alpha,\nu}'(x)\leq \sup_{x\in\R} G_{z,\alpha,\nu}'(x)<\infty.
\]
In particular,  $G_{z,\alpha,\nu}$ is strictly increasing,  Lipschitz continuous and has an inverse $G_{z,\alpha,\nu}^{-1}\colon\R\to\R$ that is Lipschitz continuous as well.
\item[(ii)] $G_{z,\alpha,\nu}(z_i)=z_i$ for all $i\in\{1, \ldots, k\}$. Moreover, $G_{z,\alpha,\nu}(x) = x$ for all $x\in(-\infty, z_1-\nu]\cup [z_k+\nu, \infty)$.
\item[(iii)]  $G_{z,\alpha,\nu}'$ is Lipschitz continuous, therefore absolutely continuous,
and it holds $G_{z,\alpha,\nu}'(z_i) = 1$ for all $i\in\{1,\dots,k\}$.
\item[(iv)] For all $i\in\{1,\dots,k+1\}$,  $G_{z,\alpha,\nu}'$ is differentiable on $(z_{i-1},z_i)$ with bounded, Lipschitz continuous derivative $G_{z,\alpha,\nu}''$.
\item[(v)] For all $i\in\{1,\dots,k\}$ the one-sided limits  $G_{z,\alpha,\nu}''(z_i-) $ and $G_{z,\alpha,\nu}''(z_i+)$ exist and satisfy
\[
G_{z,\alpha,\nu}''(z_i-) = -2\alpha_i,\quad G_{z,\alpha,\nu}''(z_i+) = 2\alpha_i.
\]
\end{itemize}
\end{lemma}

Next, assume that the coefficients $\mu$ and $\sigma$ satisfy (A1) to (A4) and note that the property (A2)(ii) of $\mu$ implies the existence of all of the limits
\[
\mu(\xi_i-)=\lim_{x\uparrow \xi_i} \mu(x),\quad \mu(\xi_i+)=\lim_{x\downarrow \xi_i} \mu(x), \quad i\in\{1,\dots,k\}.
\]
Set $\xi=(\xi_1,\dots,\xi_k)$,  define $\alpha=(\alpha_1,\dots,\alpha_k)\in \R^k$ by
\[
\alpha_i =\frac{\mu(\xi_i-)-\mu(\xi_i+)}{2 \sigma^2(\xi_i)},
\]
 for  $i\in\{1,\dots,k\}$, and let $\nu\in (0,\rho_{\xi,\alpha})$.  We define
 \begin{equation}\label{G}
 G=G_{\xi,\alpha,\nu}.
 \end{equation}

 In the following lemma we introduce and study two functions $\widetilde \mu,\widetilde \sigma\colon\R\to\R$ that are later shown to be the coefficients of the SDE~\eqref{sde0} transformed by $G$, see the proof of
  Lemma \ref{transform3} below.

\begin{lemma}\label{transform1} Let $\mu$ and $\sigma$ satisfy (A1) to (A4). Let also $G$ be given by~\eqref{G} and extend $G''\colon \cup_{i=1}^{k+1} (\xi_{i-1},\xi_i)\to \R$ to the whole real line by taking
 \begin{equation}\label{xxc}
 G''(\xi_i) = 2\alpha_i + 2\,\frac{\mu(\xi_i+)-\mu(\xi_i)}{\sigma^2(\xi_i)},
\end{equation}
for $i\in\{1, \ldots, k\}$. Then, the functions
\begin{equation}\label{tildecoeff}
\widetilde \mu=(G'\cdot \mu+\tfrac{1}{2}G''\cdot\sigma^2)\circ G^{-1} \, \text{ and }\, \widetilde\sigma=(G'\cdot\sigma)\circ G^{-1}
\end{equation}
satisfy  (A1), (A2$'$) and (A3) with $\mu$ replaced by $\widetilde \mu$ and $\sigma$ replaced by $\widetilde \sigma$.
Moreover, the  SDE
\begin{equation}\label{sde1}
	\begin{aligned}
		dZ_t & = \widetilde\mu(Z_t) \, dt + \widetilde\sigma(Z_t) \, dW_t, \quad t\in [0,\infty),\\
		Z_0 & = G(x_0),
	\end{aligned}
\end{equation}
has a unique strong solution $Z$,
 which satisfies
	\begin{equation}\label{mm2}
		\sup_{t\in[0,1]}\EE\bigl[|Z_t|^{p_0}\bigr]<\infty.
	\end{equation}
\end{lemma}

\begin{proof} Using Lemma \ref{lemx1} and the assumption that  $\mu$ and $\sigma$ satisfy (A2)(ii) and (A3) we obtain that the function $G'\cdot \mu+\tfrac{1}{2}G''\cdot\sigma^2$ is continuous on $\R\setminus\{\xi_1,\dots,\xi_k\}$. Employing Lemma \ref{lemx1} we furthermore obtain
 that for all $i\in\{1, \ldots, k\}$,
\begin{align*}
(G'\cdot \mu +\tfrac{1}{2}G''\cdot\sigma^2)(\xi_i-) & = \mu(\xi_i-) -\alpha_i\cdot \sigma^2(\xi_i)
 = (\mu(\xi_i-) +\mu(\xi_i+))/2  = (G'\cdot \mu +\tfrac{1}{2}G''\cdot\sigma^2)(\xi_i)\\
& = \mu(\xi_i+) +\alpha_i\cdot \sigma^2(\xi_i) =  (G'\cdot \mu+\tfrac{1}{2}G''\cdot\sigma^2)(\xi_i+).
\end{align*}
Hence $G'\cdot \mu+\tfrac{1}{2}G''\cdot\sigma^2$ is continuous on $\R$. Since $G^{-1}$ is continuous, see  Lemma~\ref{lemx1}, we thus obtain that $\widetilde\mu$ is continuous. Moreover,  Lemma \ref{lemx1} and the continuity of $\sigma$ imply that $\widetilde\sigma$ is continuous.

The continuity of $\widetilde\mu$ and $\widetilde\sigma$ implies that there exists  $c\in(0, \infty)$  such that, for all $x\in[\xi_1-\nu, \xi_k+\nu]$,
 \[
 2x \cdot \widetilde\mu(x)+(p_0-1)\cdot \widetilde\sigma^2(x)\leq c\leq c\cdot (1+x^2).
 \]
Note further that, for all $x\in(-\infty, \xi_1-\nu]\cup [\xi_k+\nu, \infty)$,
\begin{equation}\label{e1}
 G^{-1}(x)=x, \, G'(x)=1, \, G''(x)=0,
\end{equation}
and therefore, for all $x\in(-\infty, \xi_1-\nu]\cup [\xi_k+\nu, \infty)$,
\begin{equation}\label{vv1}
\widetilde\mu(x)=\mu(x), \, \widetilde\sigma(x)=\sigma(x).
\end{equation}
Using~\eqref{vv1} and the assumption that  $\mu$ and $\sigma$ satisfy (A1) we conclude that
 there exists  $c\in(0, \infty)$  such that, for all  $x\in(-\infty, \xi_1-\nu]\cup [\xi_k+\nu, \infty)$,
 \[
 2x \cdot \widetilde\mu(x)+(p_0-1)\cdot \widetilde\sigma^2(x)=2x \cdot \mu(x)+(p_0-1)\cdot \sigma^2(x)\leq c\cdot (1+x^2).
 \]
Thus, $\widetilde \mu$ and $\widetilde \sigma$ satisfy (A1) with $\mu$ replaced by $\widetilde \mu$ and $\sigma$ replaced by $\widetilde \sigma$.

We next show that $\widetilde \mu$ and  $\widetilde \sigma$ satisfy  (A2$'$)(ii) and (A3) with $\mu$ replaced by $
\widetilde \mu$ and $\sigma$ replaced by $\widetilde \sigma$, i.e.,  there exist $c_1, c_2\in (0, \infty)$ such that, for all $x,y\in \R$,
\begin{equation}\label{pol1}
|\widetilde\mu(x)-\widetilde\mu(y)|\leq c_1\cdot (1+|x|^{\ell_\mu}+|y|^{\ell_\mu})\cdot |x-y|
\end{equation}
and
 \begin{equation}\label{ly11}
|\widetilde\sigma(x)-\widetilde\sigma(y)|\leq c_2\cdot (1+|x|^{\ell_{\sigma}}+|y|^{\ell_{\sigma}})\cdot |x-y|.
\end{equation}
For  convenience, we use in the sequel the notation $x'=G^{-1}(x)$ for $x\in\R$.  Let $I\in\{(\xi_1-\nu, \xi_1), (\xi_1, \xi_2), \ldots, (\xi_{k-1}, \xi_k), (\xi_k, \xi_k+\nu)\}$. Clearly, for all $x,y\in I$,
\begin{equation}\label{t1}
\begin{aligned}
|& \widetilde\mu(x)-\widetilde\mu(y)|\\
&\qquad =|(G'\cdot \mu+\tfrac{1}{2}G''\cdot\sigma^2)(x')-(G'\cdot \mu+\tfrac{1}{2}G''\cdot\sigma^2)(y')|\\
&\qquad \leq |G'(x')| \cdot |\mu(x')-\mu(y')|+|G'(x')-G'(y')|\cdot |\mu(y')|\\
&\qquad\qquad+\tfrac{1}{2}|G''(x')| \cdot |\sigma(x')-\sigma(y')|\cdot |\sigma(x')+\sigma(y')| +\tfrac{1}{2}|G''(x')-G''(y')|\cdot |\sigma^2(y')|.
\end{aligned}
\end{equation}
Moreover, by Lemma~\ref{lemx1}(ii) and the fact that $G$ is strictly increasing, it follows that for all $x,y\in I$, it holds that $x', y'\in I$ as well. Using the assumption that
$\mu$ and $\sigma$ satisfy (A2)(ii) and (A3)
and Lemma \ref{lemx1} we further obtain that the functions $\mu, \sigma, G', G'', G^{-1}$ are  Lipschitz continuous on $I$ and thus, in particular,  bounded on $I$. Consequently, in view of~\eqref{t1}, we conclude  that there exist $c_1, c_2\in (0, \infty)$ such that, for all $x,y\in I$,
\begin{equation}\label{pol2}
|\widetilde\mu(x)-\widetilde\mu(y)|\leq c_1\cdot |x'-y'|\leq c_2\cdot |x-y|.
\end{equation}
Next, let $I\in\{(-\infty, \xi_1-\nu), (\xi_k+\nu, \infty)\}$.
Using \eqref{vv1} and the assumption that
$\mu$  satisfies (A2)(ii) we obtain that there exists $c\in (0, \infty)$ such that, for all $x,y\in I$,
\begin{equation}\label{pol3}
\begin{aligned}
|\widetilde\mu(x)-\widetilde\mu(y)|
=|\mu(x)-\mu(y)|\leq c\cdot (1+|x|^{\ell_\mu}+|y|^{\ell_\mu})\cdot |x-y|.
\end{aligned}
\end{equation}
Employing \eqref{pol2}, \eqref{pol3}, the continuity of $\widetilde\mu$  and the triangle inequality, we obtain \eqref{pol1}.
Proceeding in a similar way, we derive \eqref{ly11}.

Our next step is to  show that  $\widetilde \mu$ and $\widetilde \sigma$ satisfy (A2$'$)(i) with $\mu$ replaced by $\widetilde \mu$ and $\sigma$ replaced by $\widetilde \sigma$, i.e.,  there exists $c\in (0, \infty)$ such that,
 for all $x,y\in \R$,
\begin{equation}\label{pol4}
2(x-y)\cdot (\widetilde\mu(x)-\widetilde\mu(y))+(p_1-1)\cdot (\widetilde\sigma(x)-\widetilde\sigma(y))^2\leq c\cdot |x-y|^2.
\end{equation}
Put $I_1=[\xi_1-\nu-2, \xi_k+\nu+2]$.  Clearly, \eqref{pol1} and \eqref{ly11} imply that  there exists $c\in (0, \infty)$ such that, for all $x,y\in I_1$,
\begin{equation}\label{pol6}
|\widetilde\mu(x)-\widetilde\mu(y)|\leq c\cdot |x-y|\quad\text{and}\quad |\widetilde\sigma(x)-\widetilde\sigma(y)|\leq c\cdot |x-y|.
\end{equation}
It follows from \eqref{pol6}  that
\eqref{pol4} holds for all $x,y\in I_1$.
 Now, consider $I_2=[\xi_k+\nu, \infty)$. Observing~\eqref{vv1} and the assumption that  $\mu$ and $\sigma$ satisfy (A2)(i), we immediately see that
\eqref{pol4} holds for all $x,y\in I_2$ as well. Next, consider $I_3=[\xi_k+\nu+2, \infty)$ and $I_4=[\xi_1-\nu-2, \xi_k+\nu]$ and let $z=\xi_k+\nu+1$. Then, for all $x\in I_3$ and all $y\in I_4$, we have $x,z\in I_2$ and $y,z\in I_1$. Hence, there exists $c\in(0, \infty)$ such that, for all $x\in I_3$ and all $y\in I_4$,
\begin{equation}\label{g1}
\begin{aligned}
&2(x-y)\cdot(\widetilde\mu(x)-\widetilde\mu(y))\\
&\qquad=\frac{x-y}{x-z}\cdot 2(x-z)\cdot(\widetilde\mu(x)-\widetilde\mu(z))+\frac{x-y}{z-y}\cdot 2(z-y)(\widetilde\mu(z)-\widetilde\mu(y))\\
&\qquad\leq \frac{x-y}{x-z}\cdot(-(p_1-1)\cdot (\widetilde\sigma(x)-\widetilde\sigma(z))^2+c\cdot |x-z|^2)\\
&\qquad\qquad\qquad+\frac{x-y}{z-y}\cdot
(-(p_1-1)\cdot (\widetilde\sigma(z)-\widetilde\sigma(y))^2+c\cdot |z-y|^2)\\
&\qquad\leq -\frac{x-y}{x-z}\cdot(p_1-1)\cdot (\widetilde\sigma(x)-\widetilde\sigma(z))^2-\frac{x-y}{z-y}\cdot(p_1-1)\cdot (\widetilde\sigma(z)-\widetilde\sigma(y))^2\\
&\qquad\qquad\qquad + 2c\cdot |x-y|^2.
\end{aligned}
\end{equation}
Moreover, using the fact that, for all $a,b\in\R$ and all $\delta\in (0,\infty)$,
\begin{equation}\label{squares}
(a+b)^2\leq (1+\delta)\cdot a^2+ \Bigl(1+\frac{1}{\delta}\Bigr)\cdot b^2
\end{equation}
we conclude that, for all $x\in I_3$ and all $y\in I_4$,
\begin{equation}\label{g2}
\begin{aligned}
(\widetilde\sigma(x)-\widetilde\sigma(y))^2&=
 (\widetilde\sigma(x)-\widetilde\sigma(z)+\widetilde\sigma(z)-\widetilde\sigma(y))^2\\
 &\leq (1+\delta_{x,y})\cdot (\widetilde\sigma(x)-\widetilde\sigma(z))^2+\Bigl(1+\frac{1}{\delta_{x,y}}\Bigr)\cdot (\widetilde\sigma(z)-\widetilde\sigma(y))^2,
\end{aligned}
\end{equation}
where
\[
\delta_{x,y}=\frac{z-y}{x-y}.
\]
Combining \eqref{g1} and \eqref{g2} we conclude that there exists $c\in (0, \infty)$ such that, for all $x\in I_3$ and all $y\in I_4$,
\begin{align*}
&2(x-y)\cdot (\widetilde\mu(x)-\widetilde\mu(y))+(p_1-1)\cdot (\widetilde\sigma(x)-\widetilde\sigma(y))^2\\
&\qquad\qquad\leq (p_1-1)\cdot \Bigl(1+\delta_{x,y}-\frac{x-y}{x-z}\Bigr)\cdot (\widetilde\sigma(x)-\widetilde\sigma(z))^2 +(p_1-1)\cdot  (\widetilde\sigma(z)-\widetilde\sigma(y))^2\\
& \qquad\qquad
\le c\cdot |x-y|^2,
\end{align*}
where the last estimate follows from ~\eqref{pol6} and the fact that, for all $x\in I_3$ and all $y\in I_4$,
\[
1+\delta_{x,y}-\frac{x-y}{x-z}=1+\delta_{x,y}-1-\frac{z-y}{x-z}=\frac{z-y}{x-y}-\frac{z-y}{x-z}<0.
\]
Thus,
\eqref{pol4} holds for all $x\in I_3$ and all $y\in I_4$. Observing the fact that, for all $x,y\in\R$,
\begin{align*}
&2(x-y)\cdot(\widetilde\mu(x)-\widetilde\mu(y))+(p_1-1)\cdot (\widetilde\sigma(x)-\widetilde\sigma(y))^2\\
&\qquad=2(y-x)\cdot(\widetilde\mu(y)-\widetilde\mu(x))+(p_1-1)\cdot (\widetilde\sigma(y)-\widetilde\sigma(x))^2,
\end{align*}
we conclude that \eqref{pol4} holds for all $x\in I_4$ and all $y\in I_3$ as well. Consequently, \eqref{pol4} holds for all $x,y\in [\xi_1-\nu-2, \infty)$. Proceeding
similarly to the cases $x, y\in I_2$ and $(x,y)\in ( I_3\times I_4)\cup( I_4\times I_3)$
one obtains that \eqref{pol4} holds for all $x,y\in (-\infty, \xi_1-\nu]$ and for all
$(x,y)\in ([\xi_1-\nu, \infty)\times (-\infty, \xi_1-\nu-2])\cup( (-\infty, \xi_1-\nu-2]\times  [\xi_1-\nu, \infty)) $,
respectively. Hence, \eqref{pol4} holds for all $x,y\in\R$.

Finally, we turn to the SDE~\eqref{sde1}. Since the coefficients $\widetilde\mu$ and $\widetilde\sigma$ satisfy (A1), (A2$'$) and (A3) we may
 apply~\cite[Theorem 2.3.6]{Mao08} to conclude that the SDE \eqref{sde1} has a unique strong solution $Z$
and applying~\cite[Theorem 2.4.1]{Mao08} we obtain \eqref{mm2}.

The proof of the lemma is thus completed.
\end{proof}

Using Lemmas~\ref{lemx1} and~\ref{transform1} we are now able to prove the following lemma, which implies Theorem~\ref{Thm0}.

\begin{lemma}\label{transform3} Let $\mu$ and $\sigma$ satisfy (A1) to (A4). Let $G$ be given by~\eqref{G} and let $Z$ denote the unique strong solution  of the SDE~\eqref{sde1}, see Lemma~\ref{transform1}. Then $G^{-1}\circ Z$ is the unique strong solution of the SDE~\eqref{sde0}.
	\end{lemma}

\begin{proof}
Put $H=G^{-1}$.
It follows from Lemma~\ref{lemx1} that $H$ is differentiable and there exist $c_1, c_2,c_3\in(0, \infty)$ such that, for all $x,y\in\R$,
\begin{align*}
|H'(x)-H'(y)|&=\Bigl|\frac{1}{G'(G^{-1}(x))}-\frac{1}{G'(G^{-1}(y))}\Bigr|\leq c_1\cdot |G'(G^{-1}(y))-G'(G^{-1}(x))|\\
&\leq c_2\cdot |G^{-1}(y)-G^{-1}(x)|\leq c_3\cdot |x-y|.
\end{align*}
Thus, $H'$ is Lipschitz continuous, hence absolutely continuous. Moreover, Lemma~\ref{lemx1} implies that for all $i\in\{1,\dots,k+1\}$ the function $H'$ is differentiable on $(\xi_{i-1},\xi_i)$ with
\[
H''(x)=-\frac{G''(G^{-1}(x))}{(G'(G^{-1}(x)))^3}
\]
for all $x\in (\xi_{i-1},\xi_i)$.
We extend $H''$ to the whole real line by taking
$H''(\xi_i)=G''(\xi_i)$ for $i\in\{1, \ldots, k\}$, see~\eqref{xxc}.
 Applying It\^{o}'s formula to $Z$ and $H$, see, e.g.,~\cite[Problem 3.7.3]{ks91}, we conclude that for all $t\in[0,\infty)$, it holds $\PP\text{-a.s.}$ that
\begin{align*}
H(Z_t)&=H(Z_0)+\int_0^t(H'(Z_s)\cdot \widetilde\mu(Z_s)+\tfrac{1}{2}H''(Z_s)\cdot\widetilde\sigma^2(Z_s))\,ds+\int_0^t H'(Z_s)\cdot\widetilde\sigma(Z_s)\,dW_s\\
&= x_0+\int_0^t \mu(H(Z_s))\,ds +\int_0^t \sigma(H(Z_s))\,dW_s.
\end{align*}
Thus, the stochastic process $(H(Z_t))_{t\in[0,\infty)}$ is a strong solution of the SDE \eqref{sde0}.

Let $X$ be a further strong solution of the SDE \eqref{sde0}. According to Lemma~\ref{lemx1}, $G'$ is absolutely continuous.   Applying It\^{o}'s formula to $X$ and $G$ we conclude that, for all $t\in[0,\infty)$, it holds $\PP\text{-a.s.}$ that
\begin{align*}
G(X_t)&=G(x_0)+\int_0^t(G'(X_s)\cdot \mu(X_s)+\tfrac{1}{2}G''(X_s)\cdot\sigma^2(X_s))\,ds+\int_0^t G'(X_s)\cdot\sigma(X_s)\,dW_s\\
&=G(x_0)+\int_0^t \widetilde\mu(G(X_s))\,ds +\int_0^t \widetilde\sigma(G(X_s))\,dW_s.
\end{align*}
Lemma \ref{transform1} implies that  $(G(X_t))_{t\in[0,\infty)}$ and $Z$ are indistinguishable. Thus, $X$ and $(H(Z_t))_{t\in[0,\infty)}$ are indistinguishable, which yields that
  the strong solution of \eqref{sde0} is unique.
This completes the proof of the lemma.
\end{proof}

\subsection{$L_p$-estimates of the solution  and the time-continuous tamed Euler scheme}\label{4.1}
We have the following $L_p$-estimates  of the solution $X$ of the SDE~\eqref{sde0} and of its increments.
\begin{lemma}\label{solprop}
Assume (A1) to (A4). Then,
\begin{equation}\label{tmg02}
\sup_{t\in[0,1]}\EE\bigl[|X_t|^{p_0}\bigr] < \infty
\end{equation}
and for all $p\in [0,p_0)$,
\begin{equation}\label{tmg02a}
\EE\bigl[\,\sup_{t\in[0,1]}|X_t|^{p}\bigr] < \infty.
\end{equation}
Moreover, if $p_0 \ge \ell_\mu +2$,  there exists  $c\in(0, \infty)$ such that  for all $\delta\in[0,1]$ and all $s\in[0, 1-\delta]$,
\begin{equation}\label{tmg03}
\EE\Bigl[\,\sup_{t\in[s, s+\delta]} |X_t-X_s|^{2p_0/(\ell_\mu+2)}\Bigr]^{(\ell_\mu+2)/(2p_0)}\leq c\cdot \sqrt{\delta}.
\end{equation}
\end{lemma}

\begin{proof}
Recall from
Lemmas \ref{transform1} and \ref{transform3}
that, $\PP$-a.s., for all $t\in[0,1]$, we have  $X_t=G^{-1}(Z_t)$, where $G^{-1}\colon\R\to\R$
 is Lipschitz continuous and $Z$ is the unique strong solution of the SDE~\eqref{sde1} with coefficients satisfying (A1), (A2') and (A3). For the proof of Lemma~\ref{solprop} we may therefore assume that $\mu$ and $\sigma$ satisfy (A1), (A2') and (A3).
Then~\eqref{tmg02} follows from Lemma~\ref{transform1}.

We turn our attention to the proof of~\eqref{tmg02a}.
By It\^{o}'s formula and (A1) there exists $c\in (0,\infty)$ such that,  for all $t\in [0,\infty)$,
\begin{equation}\label{tmg05}
\begin{aligned}
|X_t|^{p_0} & = |x_0|^{p_0}
+\int_0^t p_0\cdot |X_u|^{p_0-2}\cdot \bigl( X_u  \mu(X_u) + \frac{1}{2}  (p_0-1) \cdot \sigma^2(X_u)\bigr)\, du \\
& \qquad +  \int_0^t p_0 \cdot X_u |X_u|^{p_0-2} \sigma(X_u)\, dW_u\\
&  \le  |x_0|^{p_0}+c\cdot \int_0^t |X_u|^{p_0-2}\cdot  (1+X_u^2)\, du +  p_0\cdot \int_0^t  X_u |X_u|^{p_0-2} \sigma(X_u)\, dW_u.
\end{aligned}
\end{equation}
Choose an increasing sequence of stopping times $(\tau_n)_{n\in\N}$ such that $\lim_{n\to\infty} \tau_n = \infty$ and the stochastic processes
\begin{equation}\label{tmg10}
Y^{(n)} = \Bigl(\int_0^{v\wedge\tau_n} X_u|X_u|^{p_0-2}\sigma(X_u)\, dW_u\Bigr)_{v\ge 0},\quad n\in\N,
\end{equation}
are martingales. Hence, for all $n\in\N$ and every bounded stopping time $\tau$,
\[
\EE[Y^{(n)}_\tau] = 0,
\]
which jointly with~\eqref{tmg05} implies that there exists $c\in (0,\infty)$ such that, for all $n\in\N$ and every stopping time $\tau$ with $ \tau \le 1$,
\begin{equation}\label{tmg06}
\EE\bigl[|X_{\tau\wedge \tau_n}|^{p_0}\bigr] \le |x_0|^{p_0}+c\cdot \EE\Bigl[ \int_0^{\tau\wedge \tau_n} |X_u|^{p_0-2}  \cdot (1+X_u^2)\, du\Bigr].
\end{equation}
Using~\cite[Lemma 3.2]{gk03b} we obtain from~\eqref{tmg06} that for all $\gamma\in (0,1)$ there exist $c_1,  c_2\in (0,\infty)$ such that, for all $n$,
\begin{equation}\label{tmg10i}
\begin{aligned}
\EE\bigl[\,\sup_{t\in[0,1]}|X_{t\wedge \tau_n}|^{\gamma p_0}\bigr] & \le \frac{2-\gamma}{1-\gamma}\cdot  \EE\Bigl[\Bigl(|x_0|^{p_0}+ c_1 \cdot\int_0^{1\wedge \tau_n} |X_u|^{p_0-2}  \cdot(1+X_u^2)\, du\Bigr)^\gamma\Bigr]\\
& \le  \frac{2-\gamma}{1-\gamma} \cdot \Bigl(|x_0|^{p_0}+c_1 \cdot\int_0^1 \EE\bigl[|X_u|^{p_0-2}  \cdot (1+X_u^2)\bigr]\, du\Bigr)^\gamma \\ & \le c_2\cdot  \Bigl(1 + \sup_{u\in[0,1]}\EE\bigl[|X_u|^{p_0}\bigr]\Bigr)^\gamma.
\end{aligned}
\end{equation}
By Fatou's lemma and~\eqref{tmg02} we conclude from~\eqref{tmg10i} that
\begin{align*}
\EE\bigl[\,\sup_{t\in[0,1]}|X_{t}|^{\gamma p_0}\bigr]  &= \EE\bigl[\liminf_{n\to\infty}\,\sup_{t\in[0,1]} |X_{t\wedge \tau_n}|^{\gamma p_0}\bigr] \le \liminf_{n\to\infty} \EE\bigl[\,\sup_{t\in[0,1]} |X_{t\wedge \tau_n}|^{\gamma p_0}\bigr] < \infty,
\end{align*}
which finishes the proof of~\eqref{tmg02a}.

We proceed with the proof of~\eqref{tmg03}.  Assume $p_0\ge \ell_\mu+2$, put $p=2p_0/(\ell_\mu+2)$ and note that $p\in[2,p_0)$. Fix $\delta\in [0,1]$ and $s\in[0,1-\delta]$.
 Throughout the following we use $c,c_1,c_2,...\in(0,\infty)$ to denote positive constants that may change their values in every appearance but neither depend on $\delta$ nor on $s$. By It\^{o}'s formula,  for all
$t\in [s,s+\delta]$,
\begin{equation}\label{tmg050}
\begin{aligned}
|X_t-X_s|^{p} & = \int_s^t p\cdot |X_u-X_s|^{p-2}\cdot \bigl( (X_u-X_s)\cdot \mu(X_u) + \frac{1}{2}  (p-1) \cdot \sigma^2(X_u)\bigr)\, du \\
& \qquad +  \int_s^t p \cdot (X_u-X_s)|X_u-X_s|^{p-2}\sigma(X_u)\, dW_u.
\end{aligned}
\end{equation}
By (A1) and~\eqref{tmg3} there exists $c\in (0,\infty)$ such that for all $u\in[s,s+\delta]$,
\begin{equation}\label{tmg05a}
\begin{aligned}
(X_u-X_s)\cdot \mu(X_u) + \frac{1}{2}  (p-1) \cdot\sigma^2(X_u) & =
 X_u\mu(X_u) + \frac{1}{2}  (p-1)\cdot \sigma^2(X_u) - X_s\mu(X_u)\\
& \le  X_u\mu(X_u) + \frac{1}{2}  (p_0-1)\cdot \sigma^2(X_u) + |X_s|\cdot |\mu(X_u)| \\
&  \le c\cdot \bigl( 1+X_u^2 + |X_s|\cdot (1 + |X_u|^{\ell_\mu+1})\bigr).
\end{aligned}
\end{equation}
Let us define
\[
Y = \sup_{s\le t\le s+\delta} |X_t-X_s|^p
\]
and note that
\begin{equation}\label{tmg05c}
\EE[Y] <\infty
\end{equation}
due to~\eqref{tmg02a} and the fact that $p<p_0$.
 Inserting~\eqref{tmg05a} into~\eqref{tmg050} yields that there exists $c\in (0,\infty)$ such that
\begin{equation}\label{tmg05b}
\begin{aligned}
Y & \le  c\cdot \,\int_s^{s+\delta} |X_u-X_s|^{p-2}\cdot  \bigl( 1+X_u^2 + |X_s|\cdot(1 + |X_u|^{\ell_\mu+1})\bigr)\, du \\
& \qquad +  \sup_{s\le t\le s+\delta}\Bigl|\int_s^t p \cdot (X_u-X_s)\cdot |X_u-X_s|^{p-2}\cdot \sigma(X_u)\, dW_u\Bigr|
\end{aligned}
\end{equation}
and employing the Burkholder-Davis-Gundy inequality and~\eqref{tmg2} we therefore conclude that there exist $c_1, c_2\in (0,\infty)$ such that
\begin{equation}\label{tmg05d}
\begin{aligned}
\EE[Y] & \le c_1\cdot  \EE\Bigl[Y^{(p-2)/p}\cdot \int_s^{s+\delta}  \bigl( 1+X_u^2 + |X_s|\cdot (1 + |X_u|^{\ell_\mu+1})\bigr)\, du\Bigr]\\& \qquad + p\cdot \EE\Bigl[\Bigl(\int_s^{s+\delta} |X_u-X_s|^{2p-2}\sigma^2(X_u)\, du\Bigr)^{1/2}\Bigr]\\
& \le c_2\cdot \EE\Bigl[Y^{(p-2)/p}\cdot \int_s^{s+\delta}  \bigl( 1+X_u^2 + |X_s|\cdot (1 + |X_u|^{\ell_\mu+1})\bigr)\, du\Bigr] \\
&  \qquad + c_1\cdot \EE\Bigl[Y^{(p-1)/p}\cdot\Bigl(\int_s^{s+\delta} (1+|X_u|^{2\ell_\sigma+2})\, du\Bigr)^{1/2}\Bigr].
\end{aligned}
\end{equation}

By the H\"older inequality,
\begin{equation}\label{tmg05e}
\begin{aligned}
 & \EE\Bigl[Y^{(p-2)/p}\cdot \int_s^{s+\delta}  \bigl( 1+X_u^2 + |X_s|\cdot (1 + |X_u|^{\ell_\mu+1})\bigr)\, du\Bigr]\\
 & \qquad\qquad \le \EE[Y]^{(p-2)/p} \cdot\EE\Bigl[\Bigl(\int_s^{s+\delta}  \bigl( 1+X_u^2 + |X_s|\cdot (1 + |X_u|^{\ell_\mu+1})\bigr)\, du\Bigr)^{p/2}\Bigr]^{2/p} \\
 & \qquad\qquad \le  \EE[Y]^{(p-2)/p} \cdot\delta^{(p-2)/p} \cdot\Bigl( \int_s^{s+\delta}  \EE\bigl[\bigl( 1+X_u^2 + |X_s|\cdot(1 + |X_u|^{\ell_\mu+1})\bigr)^{p/2}\bigr]\, du\Bigr)^{2/p}.
\end{aligned}
\end{equation}
Moreover, using~\eqref{tmg02} we obtain that there exist $c_1,c_2,c_3\in (0,\infty)$ such that
\begin{equation}\label{tmg05f}
\begin{aligned}
& \sup_{u\in[0,1]}  \EE\bigl[\bigl( 1+X_u^2 + |X_s|\cdot (1 + |X_u|^{\ell_\mu+1})\bigr)^{p/2}\bigr]\\
& \qquad \le c_1\cdot  \sup_{u\in[0,1]} \bigl(1 + \EE[|X_u|^p] + \EE[|X_s|^{p/2}\cdot (1+|X_u|^{(\ell_\mu+1)p/2})]\bigr)\\
& \qquad \le c_2 \cdot\sup_{u\in[0,1]} \bigl(1 + \EE[|X_u|^p] + \EE[|X_s|^{p_0}]^{1/(\ell_\mu+2)}\cdot\bigl(1+\EE[|X_u|^{p_0}]^{(\ell_\mu+1)/(\ell_\mu+2)}\bigr)\bigr)\leq c_3.
\end{aligned}
\end{equation}
 Combining~\eqref{tmg05e} with~\eqref{tmg05f} we see that there exists $c\in (0,\infty)$ such that
\begin{equation}\label{tmg05g}
 \EE\Bigl[Y^{(p-2)/p}\cdot \int_s^{s+\delta}  \bigl( 1+X_u^2 + |X_s|\cdot (1 + |X_u|^{\ell_\mu+1})\bigr)\, du\Bigr] \le
c \cdot \EE[Y]^{(p-2)/p} \cdot \delta.
\end{equation}

Similarly, by  the H\"older inequality, the fact that $\ell_\sigma\le \ell_\mu/2$ and \eqref{tmg02} there exist $c_1, c_2\in (0,\infty)$ such that
\begin{equation}\label{tmg05h}
\begin{aligned}
& \EE\Bigl[Y^{(p-1)/p}\cdot \Bigl(\int_s^{s+\delta} (1+|X_u|^{2\ell_\sigma+2})\, du\Bigr)^{1/2}\Bigr]\\
& \qquad\qquad \le \EE[Y]^{(p-1)/p} \cdot\delta^{(p-2)/(2p)}\cdot \Bigl( \int_s^{s+\delta}\EE\bigl[ (1+|X_u|^{2\ell_\sigma+2})^{p/2}\bigr]\, du\Bigr)^{1/p}\\
& \qquad\qquad \le c_1\cdot  \EE[Y]^{(p-1)/p} \cdot\delta^{(p-2)/(2p)}\cdot \Bigl( \int_s^{s+\delta}\bigl(1+ \sup_{u\in [0,1]}\EE[|X_u|^{p_0}]\bigr)\, du\Bigr)^{1/p}\\
& \qquad\qquad \le c_2\cdot \EE[Y]^{(p-1)/p} \cdot \sqrt \delta.
\end{aligned}
\end{equation}

Combining~\eqref{tmg05d} with~\eqref{tmg05g} and~\eqref{tmg05h} and observing~\eqref{tmg05c} we conclude that there exists $c\in [1,\infty)$ such that
\begin{equation}\label{tmg05i}
\EE[Y]^{2/p} \le c\cdot ( \delta +   \sqrt \delta\cdot \EE[Y]^{1/p}) \le (c \sqrt \delta)^2 +  2c\sqrt \delta\cdot\EE[Y]^{1/p}.
\end{equation}
 Thus,
\[
\bigl(\EE[Y]^{1/p} - c \sqrt\delta\bigr)^2 \le 2(c\sqrt\delta)^2,
\]
which yields
\[
\EE[Y]^{2/p} \le 2\bigl(\EE[Y]^{1/p} - c\sqrt\delta\bigr)^2  + 2(c\sqrt\delta)^2 \le 6 c^2\, \delta
\]
and hereby completes the proof of~\eqref{tmg03}.
\end{proof}

For later purposes we list a number of properties of the functions $\mu_n$ and $\sigma_n$ using the assumptions (A1), (A2), (A2'), (A3) on the coefficients $\mu$ and $\sigma$.

\begin{lemma}\label{tamedcoeff1}\mbox{}\\[-.5cm]
	\begin{itemize}
		\item[(i)] Assume (A2)(ii) and (A3). Then, there exists $c\in (0,\infty)$ such that, for all $n\in\N$ and $x\in\R$,
\begin{equation}
\begin{aligned}\label{tmg004}
|\sigma_n(x)|^2 & \le c\cdot\min\bigl(\sqrt{n}\,(1+x^2),\sigma^2(x)\bigr),\\
|\mu_n(x)| & \le c\cdot\min\bigl(\sqrt{n}\,(1+|x|),|\mu(x)|\bigr).
\end{aligned}
\end{equation}
	\item[(ii)] Assume (A1). Then, there exists $c\in (0,\infty)$ such that, for all $n\in\N$ and $x\in\R$,
\begin{equation}\label{tmg005}
2x\cdot\mu_n(x) + (p_0-1)\cdot (\sigma_n(x))^2 \le c\cdot (1+x^2).
\end{equation}
	\item[(iii)] Assume (A3).
Then, there exists $c\in (0,\infty)$ such that, for all $n\in\N$ and $x,y\in\R$,
\begin{equation}\label{tmg0005}
|\sigma_n(x) -\sigma_n(y)| \le c\cdot   (1+ |x|^{\ell_\sigma} + |y|^{\ell_\sigma})\cdot (n^{-1/2} +|x-y|).
\end{equation}
\item[(iv)] Assume (A2')(ii).
Then, there exists $c\in (0,\infty)$ such that, for all $n\in\N$ and $x,y\in\R$,
\begin{equation}\label{tmg0006}
	|\mu_n(x) -\mu_n(y)| \le c\cdot   (1+ |x|^{\ell_\mu} + |y|^{\ell_\mu})\cdot (n^{-1/2} +|x-y|).
	\end{equation}
\item[(v)] Assume (A2)(ii). Then, there exists $c\in (0,\infty)$ such that, for all $n\in\N$ and $x\in\R$,
\begin{equation}\label{tmg0007}
	|\mu_n(x) -\mu(x)| \le \frac{c}{\sqrt{n}} \cdot   (1+ |x|^{2\ell_\mu+1})
\end{equation}
\item[(vi)] Assume (A3). Then, there exists $c\in (0,\infty)$ such that, for all $n\in\N$ and $x\in\R$,
\begin{equation}\label{tmg0008}
	|\sigma_n(x) -\sigma(x)| \le \frac{c}{\sqrt{n}} \cdot   (1+ |x|^{\ell_\mu+\ell_\sigma + 1})\,\text{ and }\, 	|\sigma^2_n(x) -\sigma^2(x)| \le \frac{c}{\sqrt{n}} \cdot   (1+ |x|^{\ell_\mu+2\ell_\sigma + 2}).
\end{equation}
	\end{itemize}
\end{lemma}

\begin{proof}
First, we prove~\eqref{tmg004}.  The estimates $|\sigma_n|\le |\sigma|$ and $|\mu_n|\le |\mu|$ are immediate from the definition of $\mu_n$ and $\sigma_n$. Moreover, by~\eqref{tmg2} and~\eqref{tmg3} and the fact that $\ell_\sigma \le \ell_\mu/2$ we obtain that there exists $c\in (0,\infty$) such that, for all $x\in\R$ and all $n\in\N$,
\[
|\sigma_n(x)|^2 \le c\cdot \frac{(1+|x|^{\ell_\mu+2})}{(1+n^{-1/2}|x|^{\ell_\mu})^2}\quad\text{ and }\quad |\mu_n(x)| \le c\cdot \frac{(1+|x|^{\ell_\mu+1})}{1+n^{-1/2}|x|^{\ell_\mu}}.
\]
For all $x\in\R$ and all $n\in\N$,
\[
\frac{1+|x|^{\ell_\mu+2}}{(1+n^{-1/2}|x|^{\ell_\mu})^2} =n^{1/2}\Bigl(\frac{n^{-1/2}}{(1+n^{-1/2}|x|^{\ell_\mu})^2}+\frac{n^{-1/2}|x|^{\ell_\mu}}{(1+n^{-1/2}|x|^{\ell_\mu})^2}\cdot x^2\Bigr)
\le  n^{1/2} (1  +  x^2)
\]
and, similarly,
\[
\frac{1+|x|^{\ell_\mu+1}}{1+n^{-1/2}|x|^{\ell_\mu}} \le  n^{1/2} (1  +  |x|),
\]
which finishes the proof of~\eqref{tmg004}.

Clearly, \eqref{tmg005} follows from
\begin{align*}
& 2x\cdot\mu_n(x) + (p_0-1)\cdot (\sigma_n(x))^2 \\
& \qquad\qquad  = \frac{2x\cdot\mu(x)}{1+n^{-1/2}|x|^{\ell_\mu}} + \frac{(p_0-1)\cdot\sigma^2(x)}{(1+n^{-1/2}|x|^{\ell_\mu})^2} \le \frac{2x\cdot\mu(x)+(p_0-1)\cdot\sigma^2(x) }{1+n^{-1/2}|x|^{\ell_\mu}}
\end{align*}
and (A1).

Next, we prove \eqref{tmg0005}. For $n\in\N$ and $x,y\in\R$ put
\[
f_n(x,y)=\frac{n^{-1/2}(|y|^{\ell_\mu}-|x|^{\ell_\mu})\cdot \sigma(y)}{(1+n^{-1/2}|x|^{\ell_\mu})\cdot (1+n^{-1/2}|y|^{\ell_\mu})}.
\]
Using (A3) we obtain that there exists $c\in(0, \infty)$ such that, for all $n\in\N$ and all $x,y\in\R$,
\begin{equation}\label{ly3}
|\sigma_n(x)-\sigma_n(y)|=\Bigl|\frac{\sigma(x)-\sigma(y)}{1+n^{-1/2}|x|^{\ell_\mu}}+f_n(x,y)\Bigr|\leq c\cdot (1+|x|^{\ell_{\sigma}}+|y|^{\ell_{\sigma}})\cdot |x-y|+|f_n(x,y)|.
\end{equation}

We next estimate $|f_n(x,y)|$. First, assume that $\ell_\mu\in[1, \infty)$.
Clearly, for all $n\in\N$ and all $x,y\in\R$,
\begin{equation}\label{uvw3}
|f_n(x,y)| \le \frac{ ||y|^{\ell_\mu}-|x|^{\ell_\mu}|\cdot |\sigma(y)|}{1+|x|^{\ell_\mu}+|y|^{\ell_\mu}}.
\end{equation}
Since $\ell_\mu\in[1, \infty)$ we obtain there exists $c\in(0, \infty)$ such that, for all $x,y\in\R$,
\begin{equation}\label{uvw1}
	\begin{aligned}
 ||y|^{\ell_\mu}-|x|^{\ell_\mu}|&  \leq \ell_\mu\cdot |x-y|\cdot(|x|^{\ell_\mu-1}+|y|^{\ell_\mu-1}).
\end{aligned}		
\end{equation}
Moreover, employing \eqref{tmg2} and  Young's inequality we conclude that there  exist $c_1, c_2\in(0, \infty)$ such   that, for all $x,y\in\R$,
\begin{equation}\label{uvw2}
	\begin{aligned}
&(|x|^{\ell_\mu-1}+|y|^{\ell_\mu-1})\cdot |\sigma(y)| \\
&\qquad \qquad \leq c_1\cdot (1+|x|^{\ell_\mu-1}+|y|^{\ell_\mu-1})\cdot (1+|y|^{\ell_\sigma+1})\\
&\qquad \qquad=c_1\cdot(1+|x|^{\ell_\mu-1}+|y|^{\ell_\mu-1}+|y|^{\ell_\sigma+1}+|x|^{\ell_\mu-1}|y|^{\ell_\sigma+1}+|y|^{\ell_\mu+\ell_\sigma})\\
&\qquad \qquad\leq c_2\cdot(1+|x|^{\ell_\mu+\ell_\sigma}+|y|^{\ell_\mu+\ell_\sigma})\\
&\qquad \qquad\leq c_2\cdot (1+|x|^{\ell_\mu}+|y|^{\ell_\mu})\cdot (1+|x|^{\ell_\sigma}+|y|^{\ell_\sigma}).
\end{aligned}		
\end{equation}
Combining~\eqref{uvw3} to~\eqref{uvw2} yields that there exists $c\in(0, \infty)$ such that, for all $n\in\N$ and all $x,y\in\R$,
\begin{equation}\label{yl1}
|f_n(x,y)| \le c\cdot (1+|x|^{\ell_\sigma}+|y|^{\ell_\sigma})\cdot |x-y|.
\end{equation}

Next, assume that $\ell_\mu\in(0, 1)$. Observing (A3)   we obtain that  there  exists $c_1,c_2\in(0, \infty)$ such that, for all $n\in\N$ and all $x,y\in\R$ with $|x|\leq 1$ and $|y|\leq 1$,
\begin{equation}\label{yl2}
|f_n(x,y)|\leq 2\cdot n^{-1/2}\cdot |\sigma(y)| \le c_1\cdot n^{-1/2}\cdot (1+|x|^{\ell_\sigma+1}) \le c_2\cdot n^{-1/2}\cdot (1+|x|^{\ell_\sigma}).
\end{equation}
Moreover, for all $x,y\in\R$ with $|x|> 1$ or $|y|> 1$ we have
\[
\bigl||y|^{\ell_\mu}-|x|^{\ell_\mu}\bigr|\leq \frac{|y^{2}-x^{2}|}{|y|^{2-\ell_\mu}+|x|^{2-\ell_\mu}}\leq\frac{|x-y|\cdot (|x|+|y|)}{|y|^{2-\ell_\mu}+|x|^{2-\ell_\mu}}
\]
as well as
\begin{align*}
&(|y|^{2-\ell_\mu}+|x|)^{2-\ell_\mu})\cdot (1+n^{-1/2}|x|^{\ell_\mu})\cdot (1+n^{-1/2}|y|^{\ell_\mu})\\
&\qquad\qquad\geq n^{-1/2}(x^2+y^2)\geq \frac{1}{2}n^{-1/2}(|x|+|y|)^2\\
& \qquad\qquad \geq \frac{1}{4}n^{-1/2}(|x|+|y|)\cdot(1+|y|).
\end{align*}
Employing \eqref{tmg2} we thus obtain  that there  exist $c\in(0, \infty)$ such that, for all $n\in\N$ and all $x,y\in\R$ with $|x|> 1$ or $|y|> 1$
\begin{equation}\label{yl3}
	\begin{aligned}
|f_n(x,y)| & \leq n^{-1/2}||y|^{\ell_\mu}-|x|^{\ell_\mu}|\cdot |\sigma(y)|\le 4 |x-y| \cdot\frac{ 1+|y|^{\ell_\sigma+1}}{1+|y|} \leq c\cdot (1+|y|^{\ell_\sigma})\cdot |x-y|.
\end{aligned}
\end{equation}
Combining \eqref{ly3} with \eqref{yl1}, \eqref{yl2} and \eqref{yl3}  completes the proof of~\eqref{tmg0005}.

The estimate~\eqref{tmg0006} is proven in exactly the same way as the estimate~\eqref{tmg0005} by replacing $\sigma$ with $\mu$ and $\ell_\sigma$ with $\ell_\mu$.

Using~\eqref{tmg3} we obtain that there exist $c_1,c_2\in (0,\infty)$ such that for all $n\in\N$ and all $x\in \R$,
\[
|\mu_n(x)-\mu(x)| = \frac{1}{\sqrt{n}}\cdot  \frac{|x|^{\ell_\mu} |\mu(x)|}{1+n^{-1/2}|x|^{\ell_\mu}} \le  \frac{c_1}{\sqrt{n}}\cdot |x|^{\ell_\mu} (1+|x|^{\ell_\mu+1}) \le  \frac{c_2}{\sqrt{n}}\cdot  (1+|x|^{2\ell_\mu+1}),
\]
which shows ~\eqref{tmg0007}.

The first estimate in~\eqref{tmg0008} is proven in exactly the same way as ~\eqref{tmg0007} by replacing $\mu_n$ and $\mu$ with $\sigma_n$ and $\sigma$, respectively, and using~\eqref{tmg2} in place of~\eqref{tmg3}. Using the first estimate in~\eqref{tmg0008}, ~\eqref{tmg004} and~\eqref{tmg2} we obtain that there exist $c_1,c_2\in (0,\infty)$ such that for all $n\in\N$ and all $x\in \R$,
\begin{align*}
	|\sigma^2_n(x) -\sigma^2(x)|&  \le 	|\sigma_n(x) -\sigma(x)| \cdot(	|\sigma_n(x)|+|\sigma(x)|)\\
	& \le \frac{c_1}{\sqrt{n}}\cdot  (1+|x|^{\ell_\mu+\ell_\sigma+1}) \cdot (1+ |x|^{\ell_\sigma +1}) \le   \frac{c_2}{\sqrt{n}} \cdot   (1+ |x|^{\ell_\mu+2\ell_\sigma + 2}).
\end{align*}
which proves the second estimate  in~\eqref{tmg0008}.
\end{proof}

For technical reasons, we provide $L_p$-estimates and  further relevant properties
of the time-continuous tamed Euler scheme~\eqref{te} for the SDE~\eqref{sde0}
 dependent on the initial value $x_0$. To be formally precise, for every $x\in\R$, let
$X^x$ denote the unique strong solution of the SDE
\begin{equation}\label{sde00}
\begin{aligned}
dX^x_t & = \mu(X^x_t) \, dt + \sigma(X^x_t) \, dW_t, \quad t\in [0,\infty),\\
X^x_0 & = x,
\end{aligned}
\end{equation}
and, for all $x\in\R$ and $n\in\N$, let $\eul_{n}^x=(\eul_{n,t}^x)_{t\in[0,1]}$  denote the time-continuous tamed Euler scheme on $[0,1]$ with step-size $1/n$ associated to the SDE \eqref{sde00}, i.e., $\eul_{n,0}^x=x$ and
\[
\eul_{n,t}^{x}=\eul_{n,\utn}^{x}+\mu_n(\eul_{n,\utn}^{x})\cdot (t-\utn)+\sigma_n(\eul_{n,\utn}^{x})\cdot (W_t-W_{\utn})
\]
for $t\in (i/n,(i+1)/n]$ and $i\in\{0,\ldots,n-1\}$.

In particular, $X=X^{x_0}$ and, for every $n\in\N$,  $\widehat X_n = \eul_{n}^{x_0}$. Furthermore,   the integral representation
\begin{equation}\label{intrep}
\eul_{n,t}^{x}=x+\int_0^t \mu_n(\eul_{n,\usn}^{x})\, ds+\int_0^t\sigma_n(\eul_{n,\usn}^{x})\,dW_s
\end{equation}
holds for every $n\in\N$ and $t\in[0,1]$.

We have the following $L_p$-estimates  of  the time-continuous tamed Euler scheme $\eul_{n}^{x}$ and of its increments.

\begin{lemma}\label{eulprop}
Let $\mu$ and $\sigma$ satisfy (A1) to (A3). Then,  there exists $c\in (0,\infty)$ such that, for all $x\in\R$,
\begin{equation}\label{tmg002a}
\sup_{n\in\N}\sup_{t\in[0,1]}\EE\bigl[\,|\eul_{n,t}^{x}|^{p_0}\bigr]^{1/p_0} \leq c\cdot(1+|x|).
\end{equation}
Furthermore, for all $p\in [0,p_0)$  there exists $c\in (0,\infty)$ such that, for  all $x\in\R$,
\begin{equation}\label{tmg003a}
\sup_{n\in\N}\EE\bigl[\,\sup_{t\in[0,1]}|\eul_{n,t}^{x}|^{p}\bigr]^{1/p} \leq c\cdot(1+|x|).
\end{equation}
Moreover,  there exists  $c\in(0, \infty)$ such that, for  all $x\in\R$,
\begin{equation}\label{tmg003b}
	\sup_{n\in\N}\sup_{t\in[0,1]}\EE\bigl[ |\eul_{n,t}^x-\eul_{n,\utn}^x|^{p_0}\bigr]^{1/p_0}\leq c\cdot (1+|x|)\cdot \frac{1}{n^{1/4}}.
\end{equation}
Finally, if $p_0\ge \ell_\sigma+1$ then there exists  $c\in(0, \infty)$ such that
\begin{equation}\label{tmg003}
\sup_{n\in\N}\sup_{t\in[0,1]}\EE\bigl[ |\eul_{n,t}^x-\eul_{n,\utn}^x|^{p_0/(\ell_\sigma+1)}\bigr]^{(\ell_\sigma+1)/p_0}\leq c\cdot (1+|x|^{\ell_\sigma+1})\cdot \frac{1}{\sqrt{n}}.
\end{equation}
\end{lemma}

\begin{proof}
For later purposes, we  note that by using Lemma~\ref{tamedcoeff1}(i)
 and Gronwall's inequality, it is straightforward to check that for all $n\in\N$ and all $q\in [0,\infty)$ there exists $c\in (0,\infty)$, which may depend on $n$ and $q$, such that, for all $x\in\R$,
\begin{equation}\label{bbb2}
\EE\Bigl[\sup_{t\in [0,1]}|\eul^{x}_{n, t}|^q\Bigr] \leq  c\cdot (1+|x|^q).
\end{equation}

Fix $n\in\N$ and $x\in\R$. Throughout the following we  use $c,c_1,c_2,...\in(0,\infty)$ to  denote positive constants that may change their values in every appearance but neither depend on $n$ nor on $x$.

We first prove~\eqref{tmg002a}. By  It\^{o}'s formula, for all stopping times $\tau$ with $\tau \le 1$,
\begin{align*}
\begin{aligned}
|\widehat X^x_{n,\tau}|^{p_0} & = |x|^{p_0}+p_0\cdot \int_0^\tau  \bigl( \widehat X^x_{n,s}|\widehat X^x_{n,s}|^{p_0-2}\mu_n(\widehat X^x_{n,\usn}) + \frac{1}{2}  (p_0-1) \cdot|\widehat X^x_{n,s}|^{p_0-2}\sigma_n^2(\widehat X^x_{n,\usn})\bigr)\, ds \\
& \qquad +  p_0\cdot \int_0^{\tau}  \widehat X^x_{n,s}|\widehat X^x_{n,s}|^{p_0-2}\sigma_n(\widehat X^x_{n,\usn})\, dW_s.
\end{aligned}
\end{align*}
Using Lemma~\ref{tamedcoeff1}(i) and  \eqref{bbb2} we obtain that
\[
\EE\bigl[\sup_{s\in[0,1]}\bigl(\widehat X^x_{n,s}|\widehat X^x_{n,s}|^{p_0-2}\sigma_n(\widehat X^x_{n,\usn})\bigr)
^2\bigr] < \infty,
\]
which implies that the stochastic process
\[
\Bigl(\int_0^{t} \widehat X^x_{n,s}|\widehat X^x_{n,s}|^{p_0-2}\sigma_n(\widehat X^x_{n,\usn})\, dW_s\Bigr)_{t\in[0,1]}
\]
is a martingale. Thus, for all stopping times $\tau$ with $\tau \le 1$,
\begin{equation}\label{tm2}
\begin{aligned}
\EE\bigl[|\widehat X^x_{n,\tau}|^{p_0}\bigr]
&  = |x|^{p_0}+p_0\cdot \EE\Bigl[\int_0^\tau \bigl( \widehat X^x_{n,s}|\widehat X^x_{n,s}|^{p_0-2}\mu_n(\widehat X^x_{n,\usn})\\
&\qquad\qquad\qquad\qquad\qquad + \frac{1}{2}  (p_0-1) \cdot|\widehat X^x_{n,s}|^{p_0-2}\sigma_n^2(\widehat X^x_{n,\usn})\bigr)\, du\Bigr].
\end{aligned}
\end{equation}

Using  (A1) and Lemma~\ref{tamedcoeff1}(i),(ii) we obtain that there exist $c_1,c_2\in (0,\infty)$ such that,  for all $s\in[0,1]$,
\begin{equation}\label{tm3}
\begin{aligned}
& \widehat X^x_{n,s}|\widehat X^x_{n,s}|^{p_0-2}\mu_n(\widehat X^x_{n,\usn}) + \frac{1}{2}  (p_0-1) \cdot|\widehat X^x_{n,s}|^{p_0-2}\sigma_n^2(\widehat X^x_{n,\usn}) \\
& \qquad = |\widehat X^x_{n,s}|^{p_0-2} \cdot\bigl(\widehat X^x_{n,\usn}\mu_n(\widehat X^x_{n,\usn}) + \frac{p_0-1}{2}\cdot \sigma_n^2(\widehat X^x_{n,\usn})+\mu_n(\widehat X^x_{n,\usn})\cdot(\widehat X^x_{n,s}- \widehat X^x_{n,\usn})\bigr)\\
&  \qquad \le  |\widehat X^x_{n,s}|^{p_0-2} \cdot\bigl( c_1\cdot(1+|\widehat X^x_{n,\usn}|^2)\\
& \qquad\qquad\qquad +\mu_n(\widehat X^x_{n,\usn})\cdot(\mu_n(\widehat X^x_{n,\usn})\cdot(s-\usn) + \sigma_n(\widehat X^x_{n,\usn})\cdot(W_s-W_{\usn})\bigr)\\
&  \qquad \le  c_2\cdot|\widehat X^x_{n,s}|^{p_0-2}\cdot (1+|\widehat X^x_{n,\usn}|^2) +  |\widehat X^x_{n,s}|^{p_0-2}
\cdot\mu_n\sigma_n(\widehat X^x_{n,\usn})\cdot(W_s-W_{\usn})\\
&  \qquad =  c_2\cdot A_s +  B_s +  C_s,
\end{aligned}
\end{equation}
where
\begin{equation}\label{tm4}
\begin{aligned}
A_s & = |\widehat X^x_{n,s}|^{p_0-2} \cdot(1+|\widehat X^x_{n,\usn}|^2),\\
B_s & = |\widehat X^x_{n,\usn}|^{p_0-2} \cdot\mu_n\sigma_n(\widehat X^x_{n,\usn})\cdot(W_s-W_{\usn}),\\
C_s & = (|\widehat X^x_{n,s}|^{p_0-2} -|\widehat X^x_{n,\usn}|^{p_0-2}) \cdot\mu_n\sigma_n(\widehat X^x_{n,\usn})\cdot(W_s-W_{\usn}).
\end{aligned}
\end{equation}

By Young's inequality, there exist $c_1, c_2\in (0,\infty)$ such that, for all $s\in[0,1]$,
\begin{equation}\label{tm5}
\EE[A_s] \le c_1\cdot \EE\bigl[ 1+|\widehat X^x_{n,s}|^{p_0} +|\widehat X^x_{n,\usn}|^{p_0}\bigr]\le c_2\cdot \bigl(1 + \sup_{0\le u\le s}\EE\bigl[|\widehat X^x_{n,s}|^{p_0}\bigr]\bigr).
\end{equation}
Furthermore,  using the independence of $\widehat X^x_{n,\usn}$ and $W_s-W_{\usn}$,  $s\in[0,1]$, we obtain
for all $s\in[0,1]$ that
\begin{equation}\label{tm6}
\EE[B_s] =  \EE\bigl[|\widehat X^x_{n,\usn}|^{p_0-2} \mu_n\sigma_n(\widehat X^x_{n,\usn})\bigr]\cdot\EE[W_s-W_{\usn}] = 0.
\end{equation}
We finally estimate $\EE[|C_s|]$, $s\in[0,1]$. By Lemma~\ref{tamedcoeff1}(i) there exists $c\in (0,\infty)$ such that, for all $s\in[0,1]$,
\begin{equation}\label{tm7}
\begin{aligned}
|\widehat X^x_{n,s}-\widehat X^x_{n,\usn}| & = |\mu_n(\widehat X^x_{n,\usn})\cdot(s-\usn) + \sigma_n(\widehat X^x_{n,\usn})\cdot(W_s-W_{\usn})|\\
& \le c\cdot(n^{-1/2} + n^{1/4}\cdot|W_s-W_{\usn}|)\cdot(1+|\widehat X^x_{n,\usn}|).
\end{aligned}
\end{equation}
Below we show that there exists $c\in (0,\infty)$ such that,  for all $u,v,w\in\R$,
\begin{equation}\label{tm8}
\begin{aligned}
& \bigl||u+v|^{p_0-2} - |u|^{p_0-2}\bigr|\cdot |\mu_n\sigma_n(u)|\cdot|w|\\
&\qquad\qquad\le c\cdot\bigl(|v|^{p_0}\cdot (1+(n^{3/4}|w|)^{p_0}
) + 1  +|w|^{p_0/2}+  |u|^{p_0}\bigr).
\end{aligned}
\end{equation}
Applying~\eqref{tm8} with $u, v, w$ being realizations of $\widehat X^x_{n,\usn}$, $\widehat X^x_{n,s}-\widehat X^x_{n,\usn}$ and $W_s - W_{\usn}$, respectively, and observing ~\eqref{tm7}
we conclude that there exist $c_1, c_2\in (0,\infty)$ such that, for all $s\in[0,1]$,
\begin{equation}\label{tm9}
\begin{aligned}
|C_s| & \le c_1\cdot|\widehat X^x_{n,s}-\widehat X^x_{n,\usn}|^{p_0}\cdot \bigl(1+(n^{3/4}|W_s-W_{\usn}|)^{p_0}
\bigr) \\
& \qquad\qquad + c_1\cdot\bigl(1 + |W_s-W_{\usn}|^{p_0/2}+|\widehat X^x_{n,\usn}|^{p_0}\bigr)\\
& \le c_2\cdot\bigl(n^{-1/2} + n^{1/4}|W_s-W_{\usn}|\bigr)^{p_0}\bigl(1+|\widehat X^x_{n,\usn}|\bigr)^{p_0} \cdot \bigl(1+(n^{3/4}|W_s-W_{\usn}|)^{p_0}
\bigr)\\
& \qquad + c_2 \cdot \bigl(1 + |W_s-W_{\usn}|^{p_0/2}+|\widehat X^x_{n,\usn}|^{p_0}\bigr).
\end{aligned}
\end{equation}
Using  the independence of $\widehat X^x_{n,\usn}$ and $W_s-W_{\usn}$, $s\in[0,1]$, we therefore obtain that there exists $c\in (0,\infty)$ such that, for all $s\in[0,1]$,
\begin{equation}\label{tm10}
\begin{aligned}
\EE[|C_s|] & \le c\cdot\bigl(1+ \EE\bigl[|\widehat X^x_{n,\usn}|^{p_0}\bigr]\bigr)\le c\cdot \bigl(1 + \sup_{0\le u\le s}\EE\bigl[|\widehat X^x_{n,u}|^{p_0}\bigr]\bigr).
\end{aligned}
\end{equation}

Combining~\eqref{tm3} with~\eqref{tm5},~\eqref{tm6} and~\eqref{tm10} yields that there exists $c\in (0,\infty)$ such that, for all $s\in[0,1]$,
\begin{equation}\label{tm11}
\EE\Bigl[\widehat X^x_{n,s}|\widehat X^x_{n,s}|^{p_0-2}\mu_n(\widehat X^x_{n,\usn}) + \frac{1}{2} \cdot (p_0-1) |\widehat X^x_{n,s}|^{p_0-2}\sigma_n^2(\widehat X^x_{n,\usn})\Bigr]\le c\cdot \bigl(1 + \sup_{0\le u\le s}\EE\bigl[|\widehat X^x_{n,u}|^{p_0}\bigr]\bigr).
\end{equation}
 Combining the latter estimate with~\eqref{tm2} we conclude in particular that there exists $c\in (0,\infty)$ such that, for all
 $t\in[0,1]$,
\begin{equation}\label{tm12}
\sup_{0\le s\le t}\EE\bigl[|\widehat X^x_{n,s}|^{p_0}\bigr] \le |x|^{p_0} + c\cdot \int_0^t \bigl(1 + \sup_{0\le u\le s}\EE\bigl[|\widehat X^x_{n,s}|^{p_0}\bigr]\bigr)\, ds.
\end{equation}
 Observing~\eqref{bbb2} the estimate~\eqref{tmg002a} is now a consequence of the Gronwall lemma.

We proceed by proving inequality~\eqref{tm8}. Clearly,~\eqref{tm8} holds for $p_0=2$. Assume  $p_0>2$. First,  consider the case  $|u|\le 1$. Using  Lemma~\ref{tamedcoeff1}(i),~\eqref{tmg2},~\eqref{tmg3}
and the Young inequality we obtain that there exist $c_1, c_2, c_3\in (0,\infty)$ such that,  for all $u,v,w\in\R$ with  $|u|\le 1$,
\begin{align*}
& \bigl||u+v|^{p_0-2} - |u|^{p_0-2}\bigr| \cdot|\mu_n\sigma_n(u)|\cdot|w|\\
&\qquad\qquad\le c_1\cdot \bigl(|u|^{p_0-2}+|v|^{p_0-2}\bigr)\cdot \bigl(1+|u|^{\ell_\mu+\ell_\sigma+2}\bigr)\cdot|w|\\
&  \qquad\qquad \le c_{2}\cdot\bigl(1+|v|^{p_0-2}\bigr)\cdot |w| \le c_3\cdot \bigl(1+|v|^{p_0}+|w|^{p_0/2}\bigr).
\end{align*}
 Next,  consider the case $|u|> 1$. Assume first that $p_0\in (2,3]$. Put $\alpha=2/(p_0-2)$ and note that $\alpha\ge 2$. For all  $u,v\in\R$ with  $u\not=0$ we have
\begin{align*}
\bigl||u+v|^{p_0-2} - |u|^{p_0-2}\bigr| & \le \frac{\bigl||u+v|^{\alpha(p_0-2)} - |u|^{\alpha(p_0-2)}\bigr|}{|u+v|^{(\alpha-1)(p_0-2)}+|u|^{(\alpha-1)(p_0-2)}}\\
& \le \frac{\bigl||u+v|^{\alpha(p_0-2)} - |u|^{\alpha(p_0-2)}\bigr|}{|u|^{(\alpha-1)(p_0-2)}}\\
&  = \frac{\bigl||u+v|^{2} - |u|^{2}\bigr|}{|u|^{4-p_0}}\le \frac{2|v|(|u| + |v|)}{|u|^{4-p_0}}.
\end{align*}
Using Lemma~\ref{tamedcoeff1}(i) and
Young's inequality we thus conclude that there exist $c_1, c_2, c_3\in (0,\infty)$ such that,  for all $u,v,w\in\R$ with  $|u|> 1$,
\begin{align*}
& \bigl||u+v|^{p_0-2} - |u|^{p_0-2}\bigr|\cdot|\mu_n\sigma_n(u)|\cdot |w|\\
 &\qquad\qquad \le c_1\cdot \frac{|v|(|u| + |v|)}{|u|^{4-p_0}}\cdot (1+u^2)\cdot n^{3/4} |w|\\
& \qquad\qquad \le c_1\cdot \bigl(|v|^{2}n^{3/4}|w|\cdot (1+|u|^{p_0-2}) + |v|n^{3/4}|w|\cdot(1+|u|^{p_0-1})\bigr)\\
&\qquad\qquad \le c_{2}\cdot \bigl(\bigl(|v|^{2}n^{3/4}|w|\bigr)^{p_0/2} + \bigl(|v|n^{3/4}|w|\bigr)^{p_0} + 1 + |u|^{p_0}\bigr)\\
&\qquad\qquad \le c_{3}\cdot \bigl(|v|^{p_0}\cdot (1+(n^{3/4}|w|)^{p_0})  + 1 + |u|^{p_0}\bigr).
\end{align*}
Finally, assume that  $p_0\in [3,\infty)$. In this case we have for all $u, v\in\R$,
\[
\bigl||u+v|^{p_0-2} - |u|^{p_0-2}\bigr| = (p_0-2)\cdot \Bigl|\int_{|u|}^{|u+v|} y^{p_0-3}\, dy \Bigr|\le
(p_0-2)\cdot
\bigl(|u+v|^{p_0-3} + |u|^{p_0-3}\bigr)\cdot |v|.
\]
Using Lemma~\ref{tamedcoeff1}(i)
and Young's inequality we thus obtain that there exist $c_1, c_2, c_3\in (0,\infty)$ such that,  for all $u,v,w\in\R$ with  $|u|> 1$,
\begin{align*}
& \bigl||u+v|^{p_0-2} - |u|^{p_0-2}\bigr|\cdot |\mu_n\sigma_n(u)|\cdot|w|\\
& \qquad\qquad \le c_1\cdot\bigl(|v|^{p_0-3} + |u|^{p_0-3}\bigr)\cdot|v|\cdot(1+u^2)\cdot n^{3/4}|w|\\
& \qquad\qquad \le c_{2}\cdot \bigl(1+|v|^{p_0-1} + |u|^{p_0-1}\bigr)\cdot |v| n^{3/4}|w|\\
& \qquad\qquad \le c_{3}\,\bigl( 1+|v|^{p_0} + |u|^{p_0}+(|v|n^{3/4}|w|)^{p_0}\bigr)\\
& \qquad\qquad = c_{3}\cdot \bigl(|v|^{p_0}\cdot (1+(n^{3/4}|w|)^{p_0}) + 1 + |u|^{p_0}\bigr).
\end{align*}
This finishes the proof of~\eqref{tm8} and completes the proof of~\eqref{tmg002a}.

We turn to the proof of~\eqref{tmg003a}. Using~\eqref{tm2} and~\eqref{tm3} we see that there exists $c\in(0, \infty)$ such that for all stopping times $\tau$ with $\tau \le 1$,
\begin{equation}\label{o1}
 \EE\bigl[|\widehat X^x_{n,\tau}|^{p_0}\bigr]\le |x|^{p_0}+ \EE\Bigl[\int_0^\tau ( c\cdot A_s +B_s + C_s)\, ds\Bigr],
\end{equation}
where $A_s,B_s,C_s$ are given by~\eqref{tm4}.

For all $s\in[0,1]$ we put $\upsn = \lceil s n \rceil /n$.
Applying the integration by parts formula we obtain for all $i\in\{0, \ldots, n-1\}$ and  $t\in[i/n,(i+1)/n]$ that
\begin{align*}
\int_{i/n}^t (\upsn-s)\, dW_s
&=(\uptn-t)\cdot (W_t- W_{i/n})+\int_{i/n}^t (W_s-W_{\usn})\ ds.
\end{align*}
Thus, for all $t\in[0,1]$,
\begin{align*}
\int_0^t B_s\, ds  & = \int_0^t |\widehat X^x_{n,\usn}|^{p_0-2} \mu_n\sigma_n(\widehat X^x_{n,\usn})\cdot (\upsn-s)\, dW_s \\
& \qquad - (\uptn-t)\cdot (W_t-W_{\utn})\cdot |\widehat X^x_{n,\utn}|^{p_0-2} \mu_n\sigma_n(\widehat X^x_{n,\utn}).
\end{align*}
Using  Lemma~\ref{tamedcoeff1}(i) and~\eqref{bbb2} we see that
\[
\EE\Bigl[\sup_{s\in[0,1]}\bigl(|\widehat X^x_{n,\usn}|^{p_0-2} \mu_n\sigma_n(\widehat X^x_{n,\usn})\cdot (\upsn-s)\bigr)^2\Bigr] < \infty,
\]
and therefore the stochastic process
\[
\Bigl(\int_0^t |\widehat X^x_{n,\usn}|^{p_0-2} \mu_n\sigma_n(\widehat X^x_{n,\usn})\cdot (\upsn-s)\, dW_s\Bigr)_{t\in[0,1]}
\]
is a martingale. Thus, for all stopping times $\tau$ with $\tau \le 1$,
\begin{equation}\label{o3}
\EE\Bigl[\int_0^\tau B_s\, ds\Bigr] = -\EE\bigl[ (\overline\tau^n-\tau)\cdot (W_{\tau}-W_{\underline\tau_n})\cdot |\widehat X^x_{n,\underline \tau_n}|^{p_0-2} \mu_n\sigma_n(\widehat X^x_{n,\underline \tau_n})\bigr].
\end{equation}
Combining~\eqref{o1} and~\eqref{o3} we obtain that there exists $c\in(0, \infty)$ such that, for all stopping times $\tau$ with $\tau \le 1$,
\begin{equation}\label{o4}
 \EE\bigl[|\widehat X^x_{n,\tau}|^{p_0}\bigr]\le |x|^{p_0}+ c\cdot \EE\Bigl[\int_0^\tau (A_s + |C_s|)\, ds
 + n^{-1}|W_{\tau}-W_{\underline\tau_n}||\widehat X^x_{n,\underline \tau_n}|^{p_0-2} |\mu_n\sigma_n(\widehat X^x_{n,\underline \tau_n})| \Bigr].
\end{equation}
Employing~\cite[Lemma 3.2]{gk03b} we conclude from~\eqref{o4} that
for all $\gamma\in (0,1)$ there exist $c_1, c_2\in(0, \infty)$  such that
\begin{equation}\label{o5}
\begin{aligned}
\EE\bigl[\sup_{s\in[0,1]}|\widehat X^x_{n,s}|^{\gamma p_0}\bigr]
& \le \frac{2-\gamma}{1-\gamma} \cdot \EE\Bigl[\Bigl(|x|^{p_0} +c_1\cdot \int_0^1( A_s + |C_s|)\, ds \\
 &\qquad\qquad +                       n^{-1} \cdot \sup_{s\in[0,1]} |W_{s}-W_{\usn}||\widehat X^x_{n,\usn}|^{p_0-2}
  |\mu_n\sigma_n(\widehat X^x_{n,\usn})|\Bigr)^\gamma\Bigr] \\
 & \le c_2 \cdot \Bigl(|x|^{\gamma p_0} +\EE\Bigl[\Bigl(\int_0^1 (A_s + |C_s|)\, ds\Bigr)^\gamma
\Bigr] \\
 &\qquad\qquad + n^{-\gamma} \cdot \EE\Bigl[\sup_{s\in[0,1]}\bigl(  |W_{s}-W_{\usn}| |\widehat X^x_{n,\usn}|^{(p_0-2)} |\mu_n\sigma_n(\widehat X^x_{n,\usn})|\bigr)^\gamma\Bigr]\Bigr).
\end{aligned}
\end{equation}
By~\eqref{tm5},~\eqref{tm10} and~\eqref{tmg002a} for all $\gamma\in (0,1)$  there exists $c\in(0, \infty)$  such that
\begin{equation}\label{o6}
\EE\Bigl[\Bigl(\int_0^1 (A_s + |C_s|)\, ds\Bigr)^\gamma \Bigr]\le \Bigl(\int_0^1\EE[A_s + |C_s|]\, ds\Bigr)^{\gamma} \le c\cdot (1+|x|^{\gamma p_0}).
\end{equation} Moreover, using  the Burkholder-Davis-Gundy inequality and Lemma~\ref{tamedcoeff1}(i)
we obtain that for all $\gamma\in (0,1)$  there exist $c_1, c_2, c_3\in(0, \infty)$  such that
\begin{equation}\label{o7}
\begin{aligned}
& \EE\Bigl[\sup_{s\in[0,1]}\bigl(\bigl |W_{s}-W_{\usn}||\widehat X^x_{n,\usn}|^{p_0-2} |\mu_n\sigma_n(\widehat X^x_{n,\usn})|\bigr)^\gamma \Bigr] \\
& \qquad \qquad \le c_1 n^{3\gamma/4}\cdot
\EE\Bigl[\sup_{s\in[0,1]} \bigl(  |W_{s}-W_{\usn}|(1+|\widehat X^x_{n,\usn}|^{p_0} )\bigr)^\gamma \Bigr] \\
& \qquad \qquad =c_1 n^{3\gamma/4}  \cdot
\EE\Bigl[\sup_{s\in[0,1]}  \Bigl|  \int_{\usn}^s(1+|\widehat X^x_{n,\utn}|^{p_0} )\, dW_t\Bigr|^\gamma \Bigr]  \\
& \qquad \qquad \le 2c_1 n^{3\gamma/4}  \cdot
\EE\Bigl[\sup_{s\in[0,1]}  \Bigl|  \int_{0}^s(1+|\widehat X^x_{n,\utn}|^{p_0} )\, dW_t\Bigr|^\gamma \Bigr]  \\
& \qquad \qquad \le c_2n^{3\gamma/4}  \cdot
\EE\Bigl[ \Bigl(\int_{0}^1(1+|\widehat X^x_{n,\usn}|^{p_0} )^2\, ds\Bigr)^{\gamma/2} \Bigr]  \\
& \qquad \qquad \le c_3n^{3\gamma/4} \cdot\Bigl( 1 + \EE\bigl[\sup_{s\in[0,1]}|\widehat X^x_{n,s}|^{\gamma p_0}\bigr]\Bigr).
\end{aligned}
\end{equation}
 Inserting~\eqref{o6} and ~\eqref{o7} into~\eqref{o5} and observing~\eqref{bbb2} we conclude that for all $\gamma\in (0,1)$  there exist $c_1, c_2\in(0, \infty)$  such that
\begin{equation}\label{vvv3}
(1-c_1\cdot  n^{-\gamma/4}) \cdot \EE\bigl[\sup_{s\in[0,1]}|\widehat X^x_{n,s}|^{\gamma p_0}\bigr] 	\le c_2\cdot ( 1 + |x|^{\gamma p_0}).
\end{equation}
Thus, for all $\gamma\in (0,1)$  there exists $n_0\in\N$ and  $c\in(0, \infty)$  such that if $n\geq n_0$ then
\[
\EE\bigl[\sup_{s\in[0,1]}|\widehat X^x_{n,s}|^{\gamma p_0}\bigr] 	\le c\cdot ( 1 + |x|^{\gamma p_0}).
\]
Combining the latter estimate with~\eqref{bbb2} we obtain~\eqref{tmg003a}.

Finally, we turn to the proof of~\eqref{tmg003b} and~\eqref{tmg003}.
 Employing Lemma~\ref{tamedcoeff1}(i) we obtain that there exist $c_1, c_2, c_3\in(0, \infty)$  such that for all $t\in[0,1]$,
\begin{equation}\label{gg10}
	\begin{aligned}
		|\widehat{X}^x_{n,t} -\widehat{X}^x_{n,\utn}| & = |\mu_n(\widehat{X}^x_{n,\utn})\cdot (t-\utn) + \sigma_n(\widehat{X}^x_{n,\utn})\cdot (W_t-W_{\utn})   | \\
		& \le c_1\cdot\bigl(n^{1/2}\,(1+|\widehat{X}^x_{n,\utn}|)\,n^{-1}  + n^{1/4}\,(1+| \widehat{X}^x_{n,\utn}|)\cdot |W_t-W_{\utn}   | \bigr)\\
		& = c_1\cdot (1+|\widehat{X}^x_{n,\utn}|)\cdot ( n^{-1/2} + n^{1/4}|W_t-W_{\utn}   |  )
	\end{aligned}
\end{equation}
as well as
\begin{equation}\label{gg1}
	\begin{aligned}
		|\widehat{X}^x_{n,t} -\widehat{X}^x_{n,\utn}| &  \le c_2\cdot\bigl(\bigl(1+|\widehat{X}^x_{n,\utn}|\bigr)n^{-1/2}  + \bigl(1+| \widehat{X}^x_{n,\utn})  |^{\ell_\sigma +1}\bigr)\cdot |W_t-W_{\utn}   | \bigr)\\
		& \le c_3\cdot  (1+|\widehat{X}^x_{n,\utn})|^{\ell_\sigma +1}) \cdot ( n^{-1/2} + |W_t-W_{\utn}   | ).
	\end{aligned}
\end{equation}
By~\eqref{gg10} and
\eqref{tmg002a}
there exist $c_1, c_2, c_3\in(0, \infty)$  such that, for all $t\in[0,1]$,
\begin{equation}\label{gg20}
	\begin{aligned}
		\EE\bigl[ 	|\widehat{X}^x_{n,t} -\widehat{X}^x_{n,\utn}|^{p_0}\bigr]
		& \le c_1\cdot \EE\bigl[ (1+|\widehat{X}^x_{n,\utn})|^{p_0})\cdot  ( n^{-p_0/2} + n^{p_0/4}|W_t-W_{\utn}   |^{p_0}  )\bigr] \\
		& = c_1\cdot  \bigl(1+\EE\bigl[ |\widehat{X}^x_{n,\utn})|^{p_0}\bigr]\bigr) \cdot\bigl( n^{-p_0/2} +n^{p_0/4} \EE\bigl[ |W_t-W_{\utn}   |^{p_0}  \bigr]\bigr) \\
		& \le c_2\cdot  \bigl(1+\EE\bigl[ |\widehat{X}^x_{n,\utn})|^{p_0}\bigr]\bigr) n^{-p_0/4} \le c_3\cdot (1+|x|^{p_0})n^{-p_0/4},
	\end{aligned}
\end{equation}
and, similarly, by~\eqref{gg1} and
\eqref{tmg002a}
  there exist $c_1, c_2, c_3\in(0, \infty)$  such that, for all $t\in[0,1]$, with $p=p_0/(\ell_\sigma +1)$,
\begin{equation}\label{gg2}
	\begin{aligned}
\EE\bigl[ 	|\widehat{X}^x_{n,t} -\widehat{X}^x_{n,\utn}|^p\bigr]
& \le c_1\cdot  \bigl(1+\EE\bigl[ |\widehat{X}^x_{n,\utn})|^{(\ell_\sigma +1)p}\bigr]\bigr) \cdot\bigl( n^{-p/2} + \EE\bigl[ |W_t-W_{\utn}   |^p  \bigr]\bigr) \\
& \le c_2\cdot \bigl(1+\EE\bigl[ |\widehat{X}^x_{n,\utn})|^{p_0}\bigr]\bigr) n^{-p/2} \le c_3\cdot (1+|x|^{p_0})n^{-p/2}.
	\end{aligned}
\end{equation}
 This completes the proof of~\eqref{tmg003b} and~\eqref{tmg003} and finishes the proof of the lemma.

\end{proof}

\begin{rem}\label{gap}
	We add that our proof of~\eqref{tmg002a} closes a gap in the proof of Lemma 2 in~\cite{Sabanis2016} for the range $p_0\in(2,4)$.
\end{rem}

\subsection{A Markov property and occupation time estimates for the time-continuous tamed Euler scheme}\label{4.2}

The following lemma provides a Markov property of the time-continuous tamed Euler
scheme  $\eul^x_n$ relative to the gridpoints $1/n,2/n,\ldots,1$.

\begin{lemma}\label{markov}
For all $x\in\R$, all $n\in\N$, all $j\in\{0, \ldots, n-1\}$ and $\PP ^{\eul_{n,j/n}^x} $-almost all $y\in\R$ we have
\[
\PP ^{(\eul_{n,t}^x)_{t\in [j/n, 1]}|\mathcal F_{j/n}}=\PP ^{(\eul_{n,t}^x)_{t\in [j/n, 1]}|\eul_{n,j/n}^x}
\]
as well as
\[
\PP ^{(\eul_{n,t}^x)_{t\in [j/n, 1]}|\eul_{n,j/n}^x=y}=\PP ^{(\eul^y_{n,t})_{t\in [0,1-j/n]}}.
\]
\end{lemma}
\begin{proof}
The lemma is an immediate consequence of the fact that, by definition of $\eul_n^x$, for every $\ell\in \{1,\ldots,n\}$  there exists a mapping $\psi\colon \R\times C([0,\ell/n]) \to C([0,\ell/n]) $ such that for all $x\in\R$ and all $i\in\{0,1,\ldots,n-\ell\}$,
\[
(\eul^x_{n,t+i/n})_{t\in[0, \ell/n]} = \psi\bigl(\eul^x_{n, i/n},(W_{t+i/n}-W_{i/n})_{t\in[0, \ell/n]}\bigr).\qedhere
\]
\end{proof}

Next, we provide an estimate for the expected occupation time  of a neighborhood of   a non-zero of $\sigma$ by the time-continuous tamed Euler scheme $\eul_{n}^x$.

\begin{lemma}\label{occup} Assume (A1) to (A4) and $p_0\ge \ell_\mu+\ell_\sigma+2$.
Let $\xi\in\R$ satisfy $\sigma(\xi)\not=0$. Then, there exists $ c\in (0, \infty)$ such that, for   all $x\in\R$, all $n\in\N$ and all $\eps\in (0,\infty)$,
\begin{equation}
\int_0^1  \PP(\{|\eul_{n,t}^x-\xi|\leq \varepsilon\})\,dt\leq c\cdot
(1+|x|^{\ell_\mu +\ell_\sigma +2})\cdot\Bigl(\varepsilon+\frac{1}{\sqrt n}\Bigr).
\end{equation}
\end{lemma}

\begin{proof}
Let  $x\in\R$ and $n\in\N$.
By~\eqref{intrep},
\eqref{tmg004}
and Lemma~\ref{eulprop},  $\eul_n^x$ is a continuous semi-martingale with quadratic variation
\begin{equation}\label{qv}
\langle \eul_{n}^x\rangle_t
      =\int_0^t \sigma_n^2\bigl(\eul_{n,\usn}^x\bigr)\, ds,\quad t\in[0,1].
\end{equation}
For $a\in\R$ let $L^a(\eul_n^x) = (L^a_t(\eul_n^x))_{t\in[0,1]}$ denote the local time
of $\eul_{n}^x$ at the point $a$.
Thus,
for all $a\in\R$ and
 all $t\in[0,1]$,
\begin{align*}
|\eul_{n,t}^x-a| & = |x-a| + \int_0^t \sgn(\eul_{n,s}^x-a)\, \mu_n (\eul_{n,\usn}^x)\, ds\\
& \qquad\qquad  + \int_0^t \sgn(\eul_{n,s}^x-a)\, \sigma_n (\eul_{n,\usn}^x)\, dW_s + L^a_t(\eul_n^x),
\end{align*}
where $\sgn(z) = 1_{(0,\infty)}(z) - 1_{(-\infty,0]}(z)$ for $z\in\R$,
see, e.g.~\cite[Chap. VI]{RevuzYor2005}.
Hence,
for all $a\in\R$ and
 all $t\in[0,1]$,
 \begin{equation}\label{tmgx1}
\begin{aligned}
L^a_t(\eul_n^x) & \le |\eul^x_{n,t}-x| + \int_0^t |\mu_n (\eul_{n,\usn}^x)|\, ds + \Bigl|\int_0^t \sgn(\eul_{n,s}^x-a)\, \sigma_n (\eul_{n,\usn}^x)\, dW_s\Bigr|\\
& \le  \int_0^t 2 |\mu_n (\eul_{n,\usn}^x)|\, ds + \Bigl|\int_0^t\sigma_n (\eul_{n,\usn}^x)\, dW_s\Bigr|+ \Bigl|\int_0^t \sgn(\eul_{n,s}^x-a)\, \sigma_n (\eul_{n,\usn}^x)\, dW_s\Bigr|.
\end{aligned}
\end{equation}

Using the Burkholder-Davis-Gundy inequality, \eqref{tmg004}, \eqref{tmg2}, \eqref{tmg3}  and estimate~\eqref{tmg002a} in Lemma~\ref{eulprop} together with the assumption $\ell_\sigma \le \ell_\mu/2$ and the
fact that  $\max(\ell_\mu+1,\ell_\sigma+1) \le p_0$  we conclude that there exist $c_1,c_2, c_3,c_4\in (0,\infty)$ such that,
  for all $x\in\R$, all $n\in\N$,
all $a\in\R$
   and all $t\in[0,1]$,
\begin{equation}\label{local1}
\begin{aligned}
\EE\bigl[L^a_t(\eul_n^x)\bigr] &  \le  2\int_0^t \EE\bigl[|\mu_n (\eul_{n,\usn}^x)|\bigr]\, ds + c_1\cdot \EE\Bigl[\Bigl(\int_0^t \sigma^2_n (\eul_{n,\usn}^x)\, ds\Bigr)^{1/2}\Bigr]\\
& \le c_2 \cdot\Bigl(1+ \sup_{s\in[0,1]}\EE\bigl[|\eul_{n,s}^x|^{\ell_\mu+1}\bigr]
 + \sup_{s\in[0,1]}\EE\bigl[|\eul_{n,s}^x|^{\ell_\sigma+1}]\Bigr)\\
&\le c_3 \cdot\bigl(1+ |x|^{\ell_\mu+1}+|x|^{\ell_\sigma+1}\bigr) \le c_4 \cdot \bigl(1+|x|^{\ell_\mu+1}\bigr).
\end{aligned}
\end{equation}
Using~\eqref{qv} and~\eqref{local1} we obtain by the occupation
time formula that there exists $c\in (0,\infty)$ such that,
for all $x\in\R$, all $n\in\N$ and all $\eps\in (0,\infty)$,
\begin{equation}\label{local2}
\begin{aligned}
&  \EE\Bigl[\int_0^1 1_{[\xi-\eps,\xi+\eps]}(\eul^x_{n,t})\, \sigma_n^2(\eul_{n,\utn}^x)\, dt\Bigr]
\\
&\qquad\qquad = \int_{\R}1_{[\xi-\eps,\xi+\eps]}(a)\, \EE\bigl[L^a_t(\eul_n^x)\bigr]\, da
 \le c\cdot \bigl(1+ |x|^{\ell_\mu+1}\bigr)\cdot \eps.
 \end{aligned}
\end{equation}

By estimates \eqref{tmg004} and \eqref{tmg0005} in Lemma~\ref{tamedcoeff1} and~\eqref{tmg2} we see that there exists $c\in (0,\infty)$ such that, for all $n\in\N$ and all $y,z\in\R$,
\begin{equation}\label{tmg556}
\begin{aligned}
& |\sigma_n^2(z)-\sigma_n^2(y)|\\
& \qquad \le (|\sigma_n(z)|+|\sigma_n(y)|)\cdot |\sigma_n(z)-\sigma_n(y)|\le (|\sigma(z)|+|\sigma(y)|)\cdot |\sigma_n(z)-\sigma_n(y)|\\
& \qquad \le c\cdot \bigl((1 + |z|^{2\ell_\sigma+1} + |y|^{2\ell_\sigma+1}) \cdot \bigl(|z-y|+n^{-1/2}\bigr).
\end{aligned}
\end{equation}
Note  that $p_0\ge  \ell_\mu+\ell_\sigma+2$ implies $\tfrac{(2\ell_\sigma+1)p_0}{p_0-\ell_\sigma-1}\leq p_0$. Hence,
by~\eqref{tmg556} and estimates \eqref{tmg002a}  and \eqref{tmg003} in Lemma~\ref{eulprop}  we conclude that there exist $c_1, c_2, c_3\in(0, \infty)$ such that, for all $n\in\N$, all $x\in\R$ and all $t\in[0,1]$,
\begin{align*}
& \EE\bigl[|\sigma_n^2(\eul_{n,t}^x)-\sigma_n^2(\eul_{n,\utn}^x)|\bigr]\\
&\qquad\qquad  \le c_1\cdot \EE\bigl[(1 + |\eul_{n,t}^x|^{2\ell_\sigma+1} + |\eul_{n,\utn}^x|^{2\ell_\sigma+1}) \cdot\bigl(|\eul_{n,t}^x-\eul_{n,\utn}^x|+n^{-1/2}\bigr)\bigr]\\
& \qquad\qquad \le c_2\cdot \bigl(1 + \sup_{s\in[0,1]}\EE\bigl[|\eul_{n,s}^x|^{\frac{(2\ell_\sigma+1)p_0}{p_0-\ell_\sigma-1}}\bigr]^{\frac{p_0-\ell_\sigma-1}{p_0}}\bigr) \cdot \EE\bigl[ |\eul_{n,t}^x-\eul_{n,\utn}^x|^{\frac{p_0}{\ell_\sigma+1}}\bigr]^{\frac{\ell_\sigma+1}{p_0}}\\
& \qquad\qquad \qquad + c_2 \,n^{-1/2}\cdot \bigl(1 + \sup_{s\in[0,1]}\EE\bigl[|\eul_{n,s}^x|^{2\ell_\sigma+1}\bigr]\bigr)\\
&\qquad\qquad  \le c_3\,n^{-1/2}\cdot (1+|x|^{3\ell_\sigma+2}).
\end{align*}
Thus, there exists $c\in (0,\infty)$ such that, for all $n\in\N$ and all $x\in\R$,
\begin{equation}\label{local3}
\begin{aligned}
  \EE\Bigl[\int_0^1\bigl|\sigma_n^2(\eul^x_{n,t})-\sigma_n^2(\eul^x_{n,\utn})\bigr|\, dt \Bigr]
 & \le c\,n^{-1/2}\cdot (1+|x|^{3\ell_\sigma+2}).
\end{aligned}
\end{equation}

Since $\sigma$ is continuous and $\sigma(\xi)\neq 0$ there exist  $\kappa,\eps_0\in(0,\infty)$ such that
\[
\inf_{|z-\xi|< \eps_0}|\sigma(z)| \ge \kappa.
\]
Hence, for all $n\in\N$ and all $z\in (\xi-\eps_0,\xi+\eps_0)$,
\begin{equation}\label{tmgcc3}
|\sigma_n(z)| \ge \frac{|\sigma(z)|}{1 + |z|^{\ell_\mu}} \ge \frac{\kappa}{1 + (\eps_0 + |\xi|)^{\ell_\mu}}.
\end{equation}
Put $\tilde \kappa = \frac{\kappa}{1 + (\eps_0 + |\xi|)^{\ell_\mu}}$. Employing~\eqref{local2}, ~\eqref{local3} and~\eqref{tmgcc3} we conclude that there exists $c\in (0,\infty)$ such that, for all $x\in\R$, all $n\in\N$ and all $\eps \in (0,\eps_0]$,
\begin{align*}
 \int_0^1  \PP(\{|\eul_{n,t}^x-\xi|\leq \varepsilon\})\,dt  &  = \frac{1}{\tilde \kappa^2}\cdot \EE\Bigl[\int_0^1\tilde \kappa^2\, 1_{[\xi-\eps,\xi+\eps]}(\eul^x_{n,t})\, dt\Bigr] \\ &  \le \frac{1}{\tilde \kappa^2}\cdot \EE\Bigl[\int_0^1 1_{[\xi-\eps,\xi+\eps]}(\eul^x_{n,t})\,  \sigma_n^2(\eul^x_{n,t})\, dt\Bigr] \\
& \le  \frac{1}{\tilde \kappa^2}\cdot \EE\Bigl[\int_0^1\bigl( 1_{[\xi-\eps,\xi+\eps]}(\eul^x_{n,t})\,  \sigma_n^2(\eul^x_{n,\utn}) + \bigl|\sigma_n^2(\eul^x_{n,t})-\sigma_n^2(\eul^x_{n,\utn})\bigr|\bigr)\, dt\Bigr]\\
&  \le  \frac{c}{\tilde \kappa^2}\cdot (1+|x|^{\ell_\mu+1}+|x|^{3\ell_\sigma+2})\cdot \Bigl( \eps + \frac{1}{\sqrt n}\Bigr),
\end{align*}
which completes the proof of the lemma.
\end{proof}

The following lemma shows how to  transfer  the condition of a sign change of $\eul_n -\xi$ at time $t$ relative to its sign at time $\utn$ to a condition on the distance of $\eul_n$ and $\xi$ at  times $\utn-1/n, \utn-(t-\utn)$ and $\utn$.

\begin{lemma}\label{central} Assume (A1) to (A4) and
let $\xi\in\R$. Then, for all $\gamma \in (0,1/2)$ there exists $c\in(0, \infty)$ such that, for all $n\in\N$, all $0\le s\le t \le 1$ with $\utn-s \ge 1/n$ and all $A\in \F_s$,
\begin{equation}\label{central1}
\begin{aligned}
& \PP\bigl(A \cap \{(\widehat X_{n,t}-\xi)\cdot( \widehat X_{n,\utn}-\xi)\le 0\}\bigr)\\
& \qquad\qquad \le \frac{c}{n}\, \PP(A) +c\,\PP\bigl(A\cap\{\max(|\widehat X_{\utn}-\xi|,|\widehat X_{\utn-1/n}-\xi|)\geq  n^{(1/2-\gamma)/\ell_\sigma}\}\bigr) \\
& \qquad\qquad\qquad + c\, \int_{\R} \PP\bigl(A \cap \bigl\{|\widehat X_{n, \utn-(t-\utn)}-\xi| \le c\,n^{-1/2} \cdot(1+|z|)\bigr\}\bigr)\cdot e^{-\frac{z^2}{2}}\, dz.
\end{aligned}
\end{equation}
\end{lemma}

\begin{proof}
Due to \eqref{tmg2} and~\eqref{tmg004} there exists $K\in (0,\infty)$ such that, for all $n\in\N$ and all $x\in\R$,
\begin{equation}\label{tmgX1}
|\mu_n(x)|\le K \sqrt{n} \cdot(1+|x|)\quad\text{ and }\quad |\sigma_n(x)| \le K\cdot (1+|x|^{\ell_\sigma+1}).
\end{equation}
Put
\[
\kappa=2^{3\ell_\sigma+8}K\cdot  (1+|\xi|^{\ell_\sigma+1}).
\]
Let $\gamma \in (0,1/2)$ and
choose $n_0\in\N\setminus\{1\}$
such that, for all $n \geq n_0$,
\[
\kappa\,n^{-\gamma}\cdot  (1+2\sqrt{2\ln(n)}) \leq 1.
\]
Clearly, we may assume that $n\ge n_0$.

Let $ 0\le s\le t \le 1$ with $\utn-s \ge 1/n$ and let $A\in \F_s$.
If $t=\utn$ then, for all $c\in (0,\infty)$ and all $z\in\R$,
\[
 \{(\widehat X_{n,t}-\xi) \cdot(\widehat X_{n,\utn}-\xi)\le 0\} =  \{\widehat X_{n,\utn}-\xi = 0\} \subset  \bigl\{|\widehat X_{n,\utn-(t-\utn)}-\xi| \le c\,n^{-1/2} \cdot(1+|z|)\bigr\},
\]
which implies that in this case~\eqref{central1} holds for all $c\ge 1/\sqrt{2\pi}$.

Now assume that $t > \utn$ and put
\[
Z_1 = \frac{W_t-W_{\utn}}{\sqrt{t-\utn}},\quad Z_2 = \frac{W_{\utn}-W_{\utn-(t-\utn)}}{\sqrt{t-\utn}},\quad Z_3=\frac{W_{\utn-(t-\utn)}-W_{\utn-1/n} }{\sqrt{1/n-(t-\utn)}}.
\]
Below we show that
\begin{equation}\label{central2}
\begin{aligned}
\bigl\{(\widehat X_{n,t}-\xi)\cdot (\widehat X_{n,\utn}-\xi)\le 0\bigr\}&\cap \bigl\{\max_{i\in\{1,2,3\}}|Z_i| \le \sqrt{2\ln(n)}\bigr\}\\
&\cap \bigl\{\max(|\widehat X_{\utn}-\xi|,|\widehat X_{\utn-1/n}-\xi|)< n^{(1/2-\gamma)/\ell_\sigma}\bigr\}  \\
& \hspace{-3cm} \subset \bigl\{|\widehat X_{n, \utn-(t-\utn)}-\xi| \le \kappa\, n^{-1/2}\cdot (1+|Z_1|+|Z_2|)\bigr\}.
\end{aligned}
\end{equation}
Note that $Z_1,Z_2,Z_3$ are independent and identically  distributed standard normal random variables. Moreover, $(Z_1,Z_2,Z_3)$ is independent of $\F_s$ since $s \le \utn-1/n$, $(Z_1,Z_2)$ is independent of $\F_{\utn-(t-\utn)}$ and $\widehat X_{n, \utn-(t-\utn)}$ is $\F_{\utn-(t-\utn)}$-measurable. Using the latter facts jointly with~\eqref{central2} and a standard estimate of standard normal tail probabilities we  obtain that
\begin{align*}
& \PP\bigl(A \cap \{(\widehat X_{n,t}-\xi)\cdot ( \widehat X_{n, \utn}-\xi)\le 0\}\bigr)\\
& \quad \le \PP\bigl(A\cap \{|\widehat X_{n, \utn-(t-\utn)}-\xi| \le \kappa\,n^{-1/2}\cdot(1+|Z_1|+|Z_2|)\}\bigr) \\
& \qquad    + \PP\bigl(A\cap \bigl\{\max_{i\in\{1,2,3\}}|Z_i| > \sqrt{2\ln(n)}\bigr\}\bigr) + \PP\bigl(A\cap\{\max(|\widehat X_{\utn}-\xi|,|\widehat X_{\utn-1/n}-\xi|)\geq n^{(1/2-\gamma)/\ell_\sigma}\}\bigr)\\
&  \quad  \le \frac{2}{\pi}\int_{[0,\infty)^2} \PP\bigl(A\cap\bigl\{|\widehat X_{n,\utn-(t-\utn)}-\xi| \le \kappa\,n^{-1/2}\cdot(1+z_1+z_2)\bigr\}\bigr)\,e^{-\frac{z_1^2+z_2^2}{2}}\, d(z_1,z_2)\\
& \qquad   +6\PP(A)\cdot \PP\bigl(\{Z_1 > \sqrt{2\ln(n)}\}\bigr)+ \PP\bigl(A\cap\{\max(|\widehat X_{\utn}-\xi|,|\widehat X_{\utn-1/n}-\xi|)\geq n^{(1/2-\gamma)/\ell_\sigma}\}\bigr)\\
&  \quad  \le \frac{2}{\pi}\int_{\R^2} \PP\Bigl(A\cap\Bigl\{|\widehat X_{n,\utn-(t-\utn)}-\xi| \le \sqrt{2}\kappa\,n^{-1/2} \cdot\bigl(1+|\tfrac{z_1+z_2}{\sqrt{2}}|\bigr)\Bigr\}\Bigr)\, e^{-\frac{z_1^2+z_2^2}{2}}\, d(z_1,z_2)\\
& \qquad \qquad   +\frac{6\PP(A)}{\sqrt{2\pi\, 2\ln(n)}\, n}+ \PP\bigl(A\cap\{\max(|\widehat X_{\utn}-\xi|,|\widehat X_{\utn-1/n}-\xi|)\geq n^{(1/2-\gamma)/\ell_\sigma}\}\bigr)\\
&  \quad  = \frac{4}{\sqrt{2\pi}}\int_{\R} \PP\bigl(A\cap\bigl\{|\widehat X_{n,\utn-(t-\utn)}-\xi| \le \sqrt{2}\kappa \,n^{-1/2} \cdot(1+|z|)\bigr\}\bigr)\, e^{-\frac{z^2}{2}}\, dz\\
& \qquad \qquad + \frac{3\PP(A)}{\sqrt{\pi \ln(n)}\, n} + \PP\bigl(A\cap\{\max(|\widehat X_{\utn}-\xi|,|\widehat X_{\utn-1/n}-\xi|)\geq n^{(1/2-\gamma)/\ell_\sigma}\}\bigr),
\end{align*}
which yields~\eqref{central1}.

It remains to prove the inclusion~\eqref{central2}. To this end let
\begin{equation}\label{central3}
\begin{aligned}
\omega\in \bigl\{(\widehat X_{n,t}-\xi) (\widehat X_{n,\utn}-\xi)\le 0\bigr\}&\cap \bigl\{\max_{i\in\{1,2,3\}}|Z_i| \le \sqrt{2\ln(n)}\bigr\}\\ &\cap \bigl\{\max(|\widehat X_{\utn}-\xi|,|\widehat X_{\utn-1/n}-\xi|)< n^{(1/2-\gamma)/\ell_\sigma}\bigr\}.
\end{aligned}
\end{equation}
Using ~\eqref{tmgX1}  and the fact that for all $a,b\in\R$ and all $q\in[0, \infty)$,
\begin{equation}\label{JJJ}
1+|a|^q\leq 2^q\,(1+|a-b|^q)\cdot (1+|b|^q),
\end{equation}
we obtain
\begin{equation}\label{LLL}
\begin{aligned}
&|\eul_{n,\utn}(\omega) -\xi|\\
 &\qquad \le |(\eul_{n,\utn}(\omega) -\xi) - (\eul_{n,t}(\omega) -\xi)| \\
&\qquad  = |\mu_n(\eul_{n,\utn}(\omega))\cdot (t-\utn) + \sigma_n(\eul_{n,\utn}(\omega))\cdot \sqrt{t-\utn}\, Z_1(\omega)|\\
&\qquad \le K\,n^{-1/2}\cdot(1+|\eul_{n,\utn}(\omega)|) + K\,n^{-1/2}\cdot(1+|\eul_{n,\utn}(\omega)|^{\ell_\sigma+1})\cdot |Z_1(\omega)|\\
&\qquad \le 2K\,n^{-1/2}\cdot(1+|\eul_{n,\utn}(\omega)|^{\ell_\sigma+1})\cdot(1+  |Z_1(\omega)|)\\
&\qquad \le 2^{\ell_\sigma+2}K\,n^{-1/2}\cdot(1+|\eul_{n,\utn}(\omega)-\xi|^{\ell_\sigma+1})\cdot(1+|\xi|^{\ell_\sigma+1})\cdot (1+  |Z_1(\omega)|).
\end{aligned}
\end{equation}
First assume that $\ell_\sigma=0$. Using the assumption $n\geq n_0$ and \eqref{central3} we obtain
\[
4K\,n^{-1/2}\cdot(1+|\xi|)\cdot (1+  |Z_1(\omega)|)\leq \frac{1}{2}
\]
and therefore it follows from \eqref{LLL} that
\begin{equation}\label{LLL1}
|\eul_{n,\utn}(\omega) -\xi|\le 8K\,n^{-1/2}\cdot(1+|\xi|)\cdot (1+  |Z_1(\omega)|).
\end{equation}
 Next assume that $\ell_\sigma>0$.
If $|\eul_{n,\utn}(\omega) -\xi|> 1$ then \eqref{LLL} and the assumption $n \geq n_0$  imply
\begin{align*}
|\eul_{n,\utn}(\omega) -\xi|  &\le 2^{\ell_\sigma+3}K\,n^{-1/2}\cdot |\eul_{n,\utn}(\omega)-\xi|^{\ell_\sigma+1}\cdot(1+|\xi|^{\ell_\sigma+1})\cdot (1+  |Z_1(\omega)|)\\
&\leq n^{-(1/2-\gamma)}\cdot |\eul_{n,\utn}(\omega)-\xi|^{\ell_\sigma+1}.
\end{align*}
Thus,
\[
|\eul_{n,\utn}(\omega)-\xi|\geq n^{(1/2-\gamma)/{\ell_\sigma}},
\]
which is in contradiction with \eqref{central3}. Hence, $|\eul_{n,\utn}(\omega) -\xi|\leq 1$ and   \eqref{LLL} then yields
\begin{equation}\label{central4}
|\eul_{n,\utn}(\omega) -\xi|  \le 2^{\ell_\sigma+3}K\,n^{-1/2}\cdot(1+|\xi|^{\ell_\sigma+1})\cdot (1+  |Z_1(\omega)|).
\end{equation}

Similarly to \eqref{LLL}, we obtain by \eqref{tmgX1}
and \eqref{JJJ} that
\begin{equation}\label{central5}
\begin{aligned}
& |\eul_{n,\utn}(\omega)-\eul_{n,\utn-(t-\utn)}(\omega)| \\
& \qquad =
|\mu_n(\eul_{n,\utn-1/n}(\omega))\cdot (t-\utn) + \sigma_n(\eul_{n,\utn-1/n}(\omega))\, \sqrt{t-\utn}\, Z_2(\omega)|\\
& \qquad\leq 2^{\ell_\sigma+2}K\,n^{-1/2}\cdot(1+|\eul_{n,\utn-1/n}(\omega)-\xi|^{\ell_\sigma+1})\cdot(1+|\xi|^{\ell_\sigma+1})\cdot (1+  |Z_2(\omega)|).
\end{aligned}
\end{equation}
Moreover, employing  \eqref{tmgX1}, \eqref{central3} and \eqref{JJJ} we conclude that
\begin{equation}\label{central6}
\begin{aligned}
& |\eul_{n,\utn-(t-\utn)}(\omega)-\eul_{n,\utn-1/n}(\omega)|\\
& \qquad= |\mu_n(\eul_{n,\utn-1/n}(\omega))\cdot (1/n -(t-\utn)) + \sigma_n(\eul_{n,\utn-1/n}(\omega))\, \sqrt{1/n-(t-\utn)}\, Z_3(\omega)| \\
& \qquad\leq 2^{\ell_\sigma+2}K\,n^{-1/2}\cdot(1+|\eul_{n,\utn-1/n}(\omega)-\xi|^{\ell_\sigma+1})\cdot(1+|\xi|^{\ell_\sigma+1})\cdot (1+  |Z_3(\omega)|)\\
& \qquad\leq 2^{\ell_\sigma+2}K\,n^{-1/2}\cdot(1+n^{1/2-\gamma}|\eul_{n,\utn-1/n}(\omega)-\xi|)\cdot(1+|\xi|^{\ell_\sigma+1})\cdot (1+  |Z_3(\omega)|)\\
& \qquad\leq 2^{\ell_\sigma+3}K\,n^{-\gamma}\cdot(1+|\eul_{n,\utn-1/n}(\omega)-\xi|)\cdot(1+|\xi|^{\ell_\sigma+1})\cdot (1+  |Z_3(\omega)|).
\end{aligned}
\end{equation}
Since
$n\geq n_0$
we have
\[
2^{\ell_\sigma+3}K\,n^{-\gamma}\cdot(1+|\xi|^{\ell_\sigma+1})\cdot (1+  |Z_3(\omega)|) \le \frac{\kappa}{2} (1+\sqrt{2\ln(n)})\le \frac{1}{2}
 \]
and therefore~\eqref{central6} yields that
\begin{align*}
 1 +|\eul_{n,\utn-(t-\utn)}(\omega) -\xi|  &\ge 1 + |\eul_{n,\utn-1/n}(\omega)-\xi| -|\eul_{n,\utn-(t-\utn)}(\omega)-\eul_{n,\utn-1/n}(\omega)|\\
&\ge (1 + |\eul_{n,\utn-1/n}(\omega)-\xi|)/2.
\end{align*}
Thus,
\begin{equation}\label{central7}
\begin{aligned}
1 + |\eul_{n,\utn-1/n}(\omega)-\xi|^{\ell_\sigma+1}&\leq 1+(1+2\,|\eul_{n,\utn-(t-\utn)}(\omega) -\xi| )^{\ell_\sigma+1}\\
&\leq 2^{2\ell_\sigma+3} \,(1+|\eul_{n,\utn-(t-\utn)}(\omega) -\xi| ^{\ell_\sigma+1}).
\end{aligned}
\end{equation}
Using~ \eqref{LLL1}, \eqref{central4}, \eqref{central5} and~\eqref{central7} we obtain
\begin{align*}
& |\eul_{n,\utn-(t-\utn)}(\omega) -\xi|\\
 &\qquad   \le |\eul_{n,\utn}(\omega)-\eul_{n,\utn-(t-\utn)}(\omega)| +|\eul_{n,\utn}(\omega)-\xi|\\
& \qquad  \le 2^{\ell_\sigma+4}K\, n^{-1/2} \cdot (1+|\eul_{n,\utn-1/n}(\omega)-\xi|^{\ell_\sigma+1})\cdot  (1+|\xi|^{\ell_\sigma+1})\cdot  (1+|Z_1(\omega)|+|Z_2(\omega)|)\\
& \qquad  \leq 2^{3\ell_\sigma+7}K\, n^{-1/2} \cdot (1+|\eul_{n,\utn-(t-\utn)}(\omega) -\xi|^{\ell_\sigma+1})\cdot  (1+|\xi|^{\ell_\sigma+1})\cdot  (1+|Z_1(\omega)|+|Z_2(\omega)|).
\end{align*}
Arguing  similarly as for the proof of \eqref{LLL1} and \eqref{central4} we conclude that
\begin{align*}
|\eul_{n,\utn-(t-\utn)}(\omega) -\xi| &\le \kappa\, n^{-1/2}\cdot  (1+|Z_1(\omega)|+|Z_2(\omega)|).
\end{align*}
This finishes the proof of~\eqref{central2} and completes the proof of the lemma.
\end{proof}

Using Lemmas~\ref{markov},~\ref{occup} and~\ref{central} we can now establish the following two estimates on the probability of sign changes of $\eul_{n}-\xi$ relative to its sign at the gridpoints $0,1/n,\dots,1$.

\begin{lemma}\label{key}
Assume (A1) to (A4) and $p_0\ge  \ell_\mu+\ell_\sigma +2$.
Let $\xi\in\R$ satisfy $\sigma(\xi)\neq 0$ and let
\[
A_{n,t} =\{(\eul_{n,t}-\xi)\cdot(\eul_{n, \utn}-\xi)\leq 0\}
\]
for all $n\in\N$ and $t\in[0,1]$.
Then the following
two
statements hold.
\begin{itemize}
\item[(i)] There exists $c\in(0, \infty)$ such that, for all $n\in\N$, all $s\in [0,1)$ and all $A\in \F_s$,
\[
 \int_{s}^1\PP(A \cap A_{n,t})\, dt \le \frac{c}{\sqrt n} \cdot \bigl( \PP(A) + \EE\bigl[\ind_A \, |\eul_{n,\usn+1/n}-\xi|^{p_0}\bigr]\bigr).
\]
\item[(ii)] There exists $c\in(0, \infty)$ such that, for all $n\in\N$, all $s\in[0,1)$ and all $A\in \F_{s}$,
\[
\int_s^1\EE\bigl[\ind_{A\cap A_{n,t}} \, |\eul_{n,\utn+1/n}-\xi|^{p_0}\bigr]\, dt \le  \frac{c}{n^{\min(1,p_0/4)}} \cdot \bigl(\PP(A) + \EE\bigl[\ind_A \, |\eul_{n,\usn+1/n}-\xi|^{p_0}\bigr]\bigr).
\]
\end{itemize}
\end{lemma}

\begin{proof}
Let $n\in\N$, $s\in [0,1)$ and $A\in\F_s$.
In the following we use $c, c_1,c_2,\dots \in (0,\infty)$ to denote positive constants that may change their values in every appearance but  neither depend on $n$ nor on $s$ nor on $A$.

We first prove part (i) of the lemma.
Clearly we may assume that $s<1-1/n$. Then $\usn\le 1-2/n$ and we have
\begin{equation}\label{key00}
 \int_{s}^1\PP(A \cap A_{n,t})\, dt \le \frac{2}{n}\, \PP(A) +  \int_{\usn + 2/n}^1\PP(A \cap A_{n,t})\, dt.
\end{equation}
Let $\gamma\in (0, 1/4]$.
 If $t\in [\usn+2/n,1]$ then $\utn \ge \usn+2/n$, which implies $\utn-1/n \ge \usn+1/n \ge s$.
We may thus apply Lemma~\ref{central} to conclude
 that
there exists $c\in(0, \infty)$ such that
\begin{align*}
& \int_{s}^1 \PP(A\cap A_{n,t})\, dt\\
& \qquad \le
\frac{c}{n}\, \PP(A) + c\cdot\int_{\usn+2/n}^1  \PP\bigl(A\cap\{\max(|\widehat X_{n,\utn}-\xi|,|\widehat X_{n,\utn-1/n}-\xi|)\geq n^{(1/2-\gamma)/\ell_\sigma}\}\bigr)\, dt \\
& \qquad\qquad +
c\cdot \int_{\R} \int_{\usn+2/n}^1 \PP\bigl(A \cap \bigl\{|\widehat X_{n,\utn-(t-\utn)}-\xi| \le c n^{-1/2}\cdot (1+|z|)\bigr\}\bigr)\, e^{-\frac{z^2}{2}}\, dt\, dz.
\end{align*}
By the change-of-variable formula we have for all $i\in\{1, \ldots, n-1\}$ and all $\kappa\in \R$,
\[
\int_{i/n}^{(i+1)/n} \PP\bigl(A \cap \bigl\{|\widehat X_{n,\utn-(t-\utn)}-\xi| \le\kappa \bigr\}\bigr)\, dt=
\int_{(i-1)/n}^{i/n} \PP\bigl(A \cap \bigl\{|\widehat X_{n,t}-\xi| \le\kappa\bigr\}\bigr)\, dt.
\]
Moreover,
\begin{align*}
& \int_{\usn+2/n}^1  \PP\bigl(A\cap\{\max(|\widehat X_{n,\utn}-\xi|,|\widehat X_{n,\utn-1/n}-\xi|)  \geq n^{(1/2-\gamma)/\ell_\sigma}\}\bigr)\, dt \\
& \qquad\le \int_{\usn+2/n}^1  \bigl( \PP\bigl(A\cap\{|\widehat X_{n,\utn}-\xi|\geq n^{(1/2-\gamma)/\ell_\sigma}\}\bigr) + \PP\bigl(A\cap\{|\widehat X_{n,\utn-1/n}-\xi|\geq n^{(1/2-\gamma)/\ell_\sigma}\}\bigr)\bigr)\, dt \\
& \qquad\le 2\int_{\usn+1/n}^1  \PP\bigl(A\cap\{|\widehat X_{n,\utn}-\xi|\geq n^{(1/2-\gamma)/\ell_\sigma}\}\bigr)\, dt.
\end{align*}
Thus, there exists $c\in(0, \infty)$ such that
\begin{equation}\label{key01}
\begin{aligned}
& \int_{s}^1 \PP(A\cap A_{n,t})\, dt\\
& \qquad \leq \frac{c}{n}\, \PP(A) + c\cdot\int_{\usn+1/n}^1  \PP\bigl(A\cap\{|\widehat X_{n,\utn}-\xi|\geq n^{(1/2-\gamma)/\ell_\sigma}\}\}\bigr)\, dt\\
& \qquad \qquad + c\cdot \int_{\R} \int_{\usn+1/n}^{1-1/n} \PP\bigl(A \cap \bigl\{|\widehat X_{n,t}-\xi| \le c n^{-1/2}\cdot (1+|z|)\bigr\}\bigr)\, e^{-\frac{z^2}{2}}\, dt\, dz.
\end{aligned}
\end{equation}

 By the fact that $A\in \F_{\usn+1/n}$ and by
Lemma~\ref{markov} we obtain that for all $z\in\R$,
\begin{equation}\label{key02}
\begin{aligned}
& \int_{\usn+1/n}^{1-1/n} \PP\bigl(A \cap \bigl\{|\widehat X_{n,t}-\xi| \le c n^{-1/2}\cdot (1+|z|)\bigr\}\bigr)\, dt \\
& \qquad\qquad \qquad = \EE\Bigl[\ind_{A}\, \EE\Bigl[\int_{\usn+1/n}^{1-1/n} \ind_{\{|\widehat X_{n,t}-\xi| \le c n^{-1/2}\cdot (1+|z|)\}}\, dt\Bigl|\eul_{n,\usn+1/n}\Bigr]\Bigr].
\end{aligned}
\end{equation}
Moreover,
by  Lemmas~\ref{markov} and ~\ref{occup}
 we obtain that there exist $c_1, c_2\in(0, \infty)$ such that,  for all $z\in\R$
and $\PP ^{\eul_{n,\usn+1/n}} $-almost all
$x\in\R$,
\begin{equation}\label{key03}
\begin{aligned}
& \EE\Bigl[\int_{\usn+1/n}^{1-1/n} \ind_{\{|\widehat X_{n,t}-\xi| \le c n^{-1/2}\cdot (1+|z|)\}}\, dt\Bigl|\eul_{n,\usn+1/n}=x\Bigr]\\
 & \qquad\qquad = \EE\Bigl[\int_{0}^{1-2/n-\usn} \ind_{\{|\widehat X^x_{n,t}-\xi| \le c n^{-1/2}\cdot (1+|z|)\}}\, dt\Bigr]\\
 & \qquad\qquad \le c_1\cdot \bigl(1+|x|^{\ell_\mu + \ell_\sigma+2}\bigr)
 \cdot \Bigl( \frac{c}{\sqrt{n}} \cdot(1+|z|) + \frac{1}{\sqrt{n}}\Bigr)\\
 &\qquad\qquad\leq  \frac{c_2}{\sqrt n}\cdot (1+|z|)\cdot \bigl(1+|x|^{\ell_\mu + \ell_\sigma+2}\bigr).
\end{aligned}
\end{equation}
Combining~\eqref{key02} and~\eqref{key03}
we conclude that there exist $c_1, c_2\in(0, \infty)$ such that, for  all $z\in\R$,
\begin{equation}\label{key04}
\begin{aligned}
& \int_{\usn+1/n}^{1-1/n} \PP\bigl(A \cap \bigl\{|\widehat X_{n,t}-\xi| \le c n^{-1/2}\cdot (1+|z|)\bigr\}\bigr)\, dt \\
& \qquad  \le \frac{c_1}{\sqrt n}\cdot (1+|z|)\cdot \EE\bigl[\ind_{A}\, (1 + |\eul_{n,\usn+1/n}|^{\ell_\mu + \ell_\sigma+2})\bigr]\\
 & \qquad \le \frac{c_2}{\sqrt n} \cdot (1+|z|)\cdot \bigl(\PP(A) + \EE\bigl[\ind_{A}\, |\eul_{n,\usn+1/n}-\xi|^{\ell_\mu + \ell_\sigma+2}\bigr]\bigr).
\end{aligned}
\end{equation}
Hence, there exist $c_1, c_2\in(0, \infty)$ such that
\begin{equation}\label{key04a}
\begin{aligned}
& \int_{\R} \int_{\usn+1/n}^{1-1/n} \PP\bigl(A \cap \bigl\{|\widehat X_{n,t}-\xi| \le c\, n^{-1/2}\cdot (1+|z|)\bigr\}\bigr)\, e^{-\frac{z^2}{2}}\, dt\, dz\\
& \qquad \qquad \le \frac{c_1}{\sqrt n}\cdot \bigl(\PP(A) + \EE\bigl[\ind_{A}\, |\eul_{n,\usn+1/n}-\xi|^{\ell_\mu + \ell_\sigma+2}\bigr]\bigr) \cdot\int_{\R} (1+|z|)\cdot e^{-\frac{z^2}{2}}\,  dz\\
 & \qquad \qquad \le \frac{c_2}{\sqrt n}\cdot\bigl(\PP(A) + \EE\bigl[\ind_{A}\, |\eul_{n,\usn+1/n}-\xi|^{p_0}\bigr]\bigr).
\end{aligned}
\end{equation}

Next, we use $A\in \F_{\usn+1/n}$ and
Lemma~\ref{markov} to obtain
\begin{equation}\label{key04b}
\begin{aligned}
&\int_{\usn+1/n}^1  \PP\bigl(A\cap\{|\widehat X_{n,\utn}-\xi|\geq n^{(1/2-\gamma)/\ell_\sigma}\}\bigr)\, dt\\
&\qquad\qquad = \EE\Bigl[\ind_{A}\, \EE\Bigl[\int_{\usn+1/n}^{1} \ind_{\{|\widehat X_{n,\utn}-\xi|\geq n^{(1/2-\gamma)/\ell_\sigma}\}}\, dt\Bigl|\eul_{n,\usn+1/n}\Bigr]\Bigr].
\end{aligned}
\end{equation}
Moreover,
by  Lemmas~\ref{markov}, ~\ref{eulprop},
the Markov inequality
 and the fact that $\ell_\sigma/(1-2\gamma)\leq 2\ell_\sigma< p_0$
 we see that there exists $c\in(0, \infty)$ such that,  for  $\PP ^{\eul_{n,\usn+1/n}} $-almost all
$x\in\R$,
\begin{align*}
& \EE\Bigl[\int_{\usn+1/n}^{1} \ind_{\{|\widehat X_{n,\utn}-\xi|\geq n^{(1/2-\gamma)/\ell_\sigma}\}}\, dt\Bigl|\eul_{n,\usn+1/n}= x\Bigr] \\
& \qquad\qquad  =  \EE\Bigl[\int_{0}^{1-\usn-1/n}  \ind_{\{|\widehat X^x_{n,\utn}-\xi|\geq n^{(1/2-\gamma)/\ell_\sigma}\}}\, dt\Bigr]\le \int_0^1 \PP\bigl( |\widehat X^x_{n,\utn}-\xi|\geq n^{(1/2-\gamma)/\ell_\sigma}\bigr)\, dt\\
& \qquad\qquad  \leq \frac{1}{\sqrt n}\cdot \sup_{t\in[0,1]}  \EE[|\widehat X^x_{n,t}-\xi|^{\ell_\sigma/(1-2\gamma)}\bigr] \le \frac{c}{\sqrt n}\cdot  (1+|x|^{p_0}).
\end{align*}
Hence, there exist $c_1, c_2\in(0, \infty)$ such that
\begin{equation}\label{key04c}
\begin{aligned}
&\int_{\usn+1/n}^1  \PP\bigl(A\cap\{|\widehat X_{n,\utn}-\xi|\geq n^{(1/2-\gamma)/\ell_\sigma}\}\bigr)\, dt\\
&\qquad  \le \frac{c_1}{\sqrt n}\cdot \EE\bigl[1_A \, \bigl(1+ |\eul_{n,\usn+1/n}|^{p_0}\bigr)\bigr] \le \frac{c_2}{\sqrt n}\cdot \bigl(\PP(A) + \EE\bigl[1_A \,  |\eul_{n,\usn+1/n}-\xi|^{p_0}\bigr]\bigr).
\end{aligned}
\end{equation}

Combining~\eqref{key01} with~\eqref{key04a} and~\eqref{key04c} yields that there exists $c\in(0, \infty)$ such that
\begin{equation}\label{key04d}
 \int_{s}^1\PP(A \cap A_{n,t})\, dt \le \frac{c}{\sqrt n}\cdot \bigl(\PP(A) + \EE\bigl[1_A \,  |\eul_{n,\usn+1/n}-\xi|^{p_0}\bigr]\bigr),
\end{equation}
 which completes the proof of part (i) of the lemma.

We next prove part (ii) of the lemma.
Clearly,
\begin{align*}
& \int_s^{1} \EE\bigl[\ind_{A\cap A_{n,t}} \,|\eul_{n,\utn+1/n}-\xi|^{p_0}\bigr]\, dt \\
& \qquad =\int_s^{\usn+1/n}\EE\bigl[\ind_{A\cap A_{n,t}} \,|\eul_{n,\utn+1/n}-\xi|^{p_0}\bigr]\, dt+ \int_{\usn+1/n}^1\EE\bigl[\ind_{A\cap A_{n,t}} \,|\eul_{n,\utn+1/n}-\xi|^{p_0}\bigr]\, dt.
\end{align*}

If $t\in[s, \usn+1/n)$ then $\underline t_n=\underline s_n$ and therefore
\begin{equation}\label{key003}
\begin{aligned}
 \int_s^{\usn+1/n} \EE\bigl[\ind_{A\cap A_{n,t}}  \,|\eul_{n,\utn+1/n}-\xi|^{p_0}\bigr]\, dt&=\int_s^{\usn+1/n}\EE\bigl[\ind_{A\cap A_{n,t}}\,|\eul_{n,\usn+1/n}-\xi|^{p_0} \bigr]\, dt\\
& \le  \int_s^{\usn+1/n}  \EE\bigl[\ind_A \,|\eul_{n,\usn+1/n}-\xi|^{p_0}\bigr]\, dt
\\ &  \le \frac{1}{n}\,  \EE\bigl[\ind_A\,|\eul_{n,\usn+1/n}-\xi|^{p_0}\bigr].
\end{aligned}
\end{equation}

Next, let $t\in [\usn+1/n, 1]$.  Clearly, we have on $A_{n,t}$,
\begin{align*}
|\eul_{n,\utn+1/n}-\xi| &\le |\eul_{n,\utn+1/n}-\eul_{n,t}| + |\eul_{n,t}-\xi| \le |\eul_{n,\utn+1/n}-\eul_{n,t}| + |\eul_{n,t}-\eul_{n,\utn}|.
\end{align*}
Hence, by Lemma~\ref{markov} and the fact that $A\in  \mathcal F_{\usn+1/n}$,
\begin{equation}\label{key05}
\begin{aligned}
& \EE\bigl[\ind_{A\cap A_{n,t}}\, |\eul_{n,\utn+1/n} - \xi|^{p_0}\bigr] \\ &  \qquad\quad  \le \EE \bigl[\ind_A \, (|\eul_{n,\utn+1/n}-\eul_{n,t}| + |\eul_{n,t}-\eul_{n,\utn}|)^{p_0} \bigr]\\ & \qquad\quad = \EE\bigl[ \ind_A \cdot \EE\bigl[(|\eul_{n,\utn+1/n}-\eul_{n,t}| + |\eul_{n,t}-\eul_{n,\utn}|)^{p_0}\bigl| \eul_{n,\usn+1/n}\bigr]\bigr].
\end{aligned}
\end{equation}
If $t\ge \usn+1/n$ then $\utn \ge \usn+1/n$. Hence, by Lemma~\ref{markov} and estimate~\eqref{tmg003b}
in Lemma~\ref{eulprop} we obtain that there exist $c_1, c_2\in(0, \infty)$ such that,  for all  $t\in [\usn+1/n, 1]$ and
$\PP ^{\eul_{n,\usn+1/n}} $-almost all
$x\in\R$,
\begin{equation}\label{key06}
\begin{aligned}
&  \EE\bigl[(|\eul_{n,\utn+1/n}-\eul_{n,t}| + |\eul_{n,t}-\eul_{n,\utn}|)^{p_0}\bigl| \eul_{n,\usn+1/n} = x\bigr]\\
& \qquad \quad = \EE\bigl[(|\eul^x_{n,\utn-\usn}-\eul^x_{n,t-\usn-1/n}| + |\eul^x_{n,t-\usn-1/n}-\eul^x_{\utn-\usn-1/n}|)^{p_0}\bigr]
\\& \qquad\quad \le  \frac{c_1}{n^{p_0/4}}\cdot (1+|x|^{p_0})
\le  \frac{c_2}{n^{p_0/4}}\cdot (1+|x-\xi|^{p_0}).
\end{aligned}
\end{equation}
It follows from~\eqref{key05} and~\eqref{key06} that there exists $c\in(0, \infty)$ such that
\begin{equation}\label{key07}
\begin{aligned}
& \int_{\usn+1/n}^1\EE\bigl[\ind_{A\cap A_{n,t}}\, |\eul_{n,\utn+1/n} - \xi|^{p_0}\bigr]\, dt \\
& \qquad\qquad  \le \frac{c}{n^{p_0/4}}
\, \int_{\usn+1/n}^1 \EE\bigl[\ind_A\, (1+|\eul_{n,\usn+1/n}-\xi|^{p_0}\bigr]\, dt\\
& \qquad\qquad \le \frac{c}{n^{p_0/4}}
\cdot \bigl(\PP(A) + \EE\bigl[\ind_A\, |\eul_{n,\usn+1/n}-\xi|^{p_0}\bigr]\bigr).
\end{aligned}
\end{equation}
Combining~\eqref{key003} with~\eqref{key07} completes the proof of part (ii) of the lemma.
\end{proof}

We are ready to establish the main result in this section, which provides a $p$-th mean estimate of the
Lebesgue measure of the set
of times $t$ of a sign change of $\eul_{n,t}-\xi$ relative to the sign of $\eul_{n,\utn}-\xi$.

\begin{prop}\label{prop1}
Assume (A1) to (A4) and $p_0\ge \ell_\mu + \ell_\sigma +2$.
Let $\xi\in\R$ satisfy $\sigma(\xi)\not=0$ and let
 $p\in [1,\infty)$.
Then there exists  $c\in(0, \infty)$ such that, for all $n\in\N$,
\begin{equation}\label{l33}
\EE\Bigl[\Bigl|\int_0^1  \ind_{\{(\eul_{n,t}-\xi)\cdot(\eul_{n, \utn}-\xi)\leq 0\}}\,dt\Bigr|^p\Bigr]^{1/p}\leq c\,n^{-1/2}.
\end{equation}
\end{prop}

\begin{proof}
Clearly, it suffices to consider only the case $p\in\N$.
For $n\in\N$ and $t\in [0,1]$
let $A_{n,t} =\{(\eul_{n,t}-\xi)\,(\eul_{n, \utn}-\xi)\leq 0\}$   as in Lemma~\ref{key} and, for $n,p\in\N$, let
\[
a_{n,p} = \EE\Bigl[\Bigl(\int_0^1 \ind_{A_{n,t}}\, dt\Bigr)^p\Bigr].
\]
 We prove by induction on $p$ that for every $p\in\N$ there exists $c\in (0,\infty)$ such that, for all $n\in\N$,
\begin{equation}\label{prop01}
a_{n,p} \le c\, n^{-p/2}.
\end{equation}

Using Lemma~\ref{key}(i) with $s=0$ and $A=\Omega$ as well as estimate~\eqref{tmg002a} in Lemma~\ref{eulprop} we
obtain that there exist $c_1,c_2\in (0,\infty)$ such that, for all $n\in\N$,
\[
a_{n,1} = \int_0^1 \PP(A_{n,t})\, dt  \le \frac{c_1}{\sqrt n}\cdot (1 + \EE[|\eul_{n,1/n}-\xi|^{p_0} ]) \le \frac{c_2}{\sqrt n}.
\]
Thus, ~\eqref{prop01} holds for $p=1$.

Next, let $q\in\N$
and assume that~\eqref{prop01} holds for all $p\in\{1, \ldots, q\}$.
Clearly,
for all $n\in\N$,
\begin{align*}
a_{n,q+1}  &= (q+1)!\cdot \int_0^1\int_{t_1}^1\ldots \int_{t_q}^1 \PP(A_{n,t_1}\cap A_{n,t_2}\cap \ldots \cap A_{n,t_{q+1}})\,dt_{q+1}\, \ldots \,dt_2\, dt_1.
\end{align*}
Note that $p_0\ge \ell_\mu + \ell_\sigma +2$ implies that $\min(1,p_0/4)\ge 1/2$. Hence, by
first applying Lemma~\ref{key}(i) with $A= A_{n,t_1}\cap \ldots \cap A_{n,t_q}$ and $s= t_{q}$, then applying $(q-1)$-times Lemma~\ref{key}(ii) with $A= A_{n,t_1}\cap \ldots \cap A_{n,t_j}$ and $s= t_j$ for $j=q-1,\ldots, 1$, and finally applying Lemma~\ref{key}(ii) with $A=\Omega$ and $s=0$, and observing the estimate~\eqref{tmg002a} in Lemma~\ref{eulprop} we conclude that there exist  $c_1, c_2, c_3\in(0, \infty)$ such that, for all $n\in\N$,
\begin{align*}
a_{n,q+1} & \le  \frac{c_1}{\sqrt n}\cdot \Bigl( a_{n,q}+\int_0^1\ldots \int_{t_{q-1}}^1 \EE\bigl[\ind_{A_{n,t_1}\cap \ldots \cap A_{n,t_q}}\, |\eul_{n,\underline{t_q}_n+1/n}-\xi|^{p_0} ]\,dt_q \, \ldots \, dt_1 \Bigr)\\
&  \le c_2\cdot \Bigl(\frac{a_{n,q}}{\sqrt n}  + \frac{a_{n,q-1}}{n}+\ldots + \frac{a_{n,1} }{n^{q/2}}+ \frac{1}{n^{q/2}}\,\int_0^1  \EE\bigl[\ind_{A_{n,t_1}}\, |\eul_{n,\underline{t_1}_n+1/n} - \xi|^{p_0} \bigr] \, dt_1\Bigr)\\
&  \le c_2\cdot\Bigl(\frac{a_{n,q}}{\sqrt n}  + \frac{a_{n,q-1}}{n}+\ldots + \frac{a_{n,1} }{n^{q/2}}+ \frac{c_3}{n^{(q+1)/2}} \Bigr).
\end{align*}
Employing  the induction hypothesis yields the validity of~\eqref{prop01} for $p=q+1$, which finishes the proof of the proposition.
\end{proof}

\subsection{Proof of Theorem~\ref{Thm1}}\label{proof1}
Let $G$ be given by~\eqref{G} and consider the associated SDE \eqref{sde1} with initial value $G(x_0)$, coefficients $\widetilde\mu$ and $\widetilde\sigma$ given by~\eqref{tildecoeff} and solution $Z$.

For every $n\in\N$ we define a corresponding time-continuous tamed Euler scheme $\eultr_{n}=(\eultr_{n,t})_{t\in[0,1]}$  on $[0,1]$ with step-size $1/n$ by $\eultr_{n,0}=G(x_0)$ and
\begin{equation}\label{eultr}
\eultr_{n,t}=\eultr_{n,i/n}+ \widetilde\mu_n(\eultr_{n,i/n})\cdot (t-i/n)+ \widetilde\sigma_n(\eultr_{n,i/n})\cdot (W_t-W_{i/n})
\end{equation}
for $t\in(i/n,(i+1)/n]$ and $i\in\{0,\ldots,n-1\}$, where
\[
\widetilde\mu_n(x) = \frac{\widetilde\mu(x)}{1+n^{-1/2}|x|^{\ell_\mu}}\quad \text{ and }\quad \widetilde\sigma_n(x) = \frac{\widetilde\sigma(x)}{1+n^{-1/2}|x|^{\ell_\mu}}
\]
for every $x\in\R$.

 Lemma~\ref{transform3} and the Lipschitz continuity of $G^{-1}$ yields the existence of $c\in(0,\infty)$ such that, for all $n\in\N$ and all $t\in[0,1]$,
\begin{equation}\label{qq1}
	|X_t-\widehat X_{n,t}| \le c\cdot 	|Z_t- G(\widehat X_{n,t})|.
\end{equation}

Employing Lemma~\ref{transform1}, the following estimate of the error of $\widehat Z_n$  is a straightforward consequence of~\cite[Theorem 3]{Sabanis2016}.

\begin{theorem}\label{sotirios-aap}\mbox{}
 Let $\mu$ and $\sigma$ satisfy (A1) to (A4) with  $p_0 > 4\ell_\mu+2$ and $p_1>2$.
	Then, for every $p\in (0,p_1)\cap (0,\frac{p_0}{2\ell_\mu +1})$ there exists $c\in (0,\infty)$ such that, for all $n\in\N$,
	\[
		\EE[\|Z_t-\widehat Z_{n,t}\|_\infty^p]^{1/p}	\le  c/\sqrt{n}.	
	\]
	
\end{theorem}

For every $n\in\N$ we define a stochastic process $Y_n =(Y_{n,t})_{t\in[0,1]}$ by
\[
Y_{n,t} = G(\widehat X_{n,t}) - \widehat Z_{n,t},\quad t\in[0,1].
\]
Below we show the following two moment estimate for the processes $Y_n$.
\begin{theorem}\label{comp}\mbox{}
 Let $\mu$ and $\sigma$ satisfy (A1) to (A4) with $p_0 > 2(\ell_\mu +  \max(\ell_\mu, 2\ell_\sigma +2)+ 1)$ and $p_1>2$.
		Then, for every $p\in (0,p_1)\cap (0,\frac{p_0}{\ell_\mu +  \max(\ell_\mu, 2\ell_\sigma +2)+ 1})$
		there exists $c\in (0,\infty)$ such that, for all $n\in\N$,
		\[
			\EE[\|Y_{n}\|_\infty^p]^{1/p}	\le  c/\sqrt{n}.	
		\]
\end{theorem}
Combining~\eqref{qq1},
Theorem~\ref{sotirios-aap} and Theorem~\ref{comp} yields Theorem~\ref{Thm1}.

It remains to prove Theorem~\ref{comp}. To this end we first note that, by Lemma~\ref{transform1}, the estimates in Lemma~\ref{solprop}, Lemma~\ref{tamedcoeff1}, and Lemma~\ref{eulprop} as well as~\eqref{bbb2} do also hold for the process $Z$, the tamed coefficients $\widetilde \mu_n$, $\widetilde \sigma_n$ and the tamed Euler scheme $\widehat Z_n$, respectively.

Since $G'$ is absolutely continuous, see Lemma~\ref{lemx1}(iii) and \eqref{G}, we may apply the It\^{o} formula, see e.g.~\cite[Problem 3.7.3]{ks91}, to obtain that $\PP\text{-a.s.}$ for all $t\in[0,1]$,
\begin{align*}
	G(\eul_{n,t}) & = G(x_0)+\int_0^t(G'(\eul_{n,s})\cdot \mu_n(\eul_{n,\usn})+\tfrac{1}{2}G''(\eul_{n,s})\cdot\sigma_n^2(\eul_{n,\usn}))\,ds\\
	& \qquad +\int_0^t G'(\eul_{n,s})\cdot\sigma_n(\eul_{n,\usn})\,dW_s.
\end{align*}
It follows that for all $n\in\N$ we have $\PP\text{-a.s.}$ for all $t\in[0,1]$,
\begin{equation}\label{eq:diff_G_and_Z}
	\begin{aligned}
	Y_{n,t}= &
	\int_0^t  \bigl(G'(\eul_{n,s})\cdot  \mu_n(\eul_{n,\usn}) - \widetilde \mu_n(\eultr_{n,\usn})  +\tfrac{1}{2}G''(\eul_{n,s})\cdot\sigma_n^2(\eul_{n,\usn})\bigr)\,ds\\
		& \qquad +\int_0^t \bigl(G'(\eul_{n,s})\cdot\sigma_n(\eul_{n,\usn}) -\widetilde \sigma_n(\eultr_{n,\usn})\bigr) \,dW_s.
		\end{aligned}
\end{equation}

For  $\alpha\in (0,\infty)$ and  $q\in [2, \infty)$ we define
\[
V_{\alpha,q} \colon [0,\infty) \times \R \to \R, \, (t,y) \mapsto \exp(-\alpha\cdot t)\cdot |x|^q.
\]
Moreover, we put
\[
\kappa = \ell_\mu + \max(\ell_\mu, 2\ell_\sigma +2)+ 1
\]
and we fix
\begin{equation}\label{parp}
p\in [2,p_1)\cap [2,\frac{p_0}{\kappa}].
\end{equation}
By the It\^{o} formula, for $\alpha\in(0, \infty)$,
all $n\in\N$ and all $t\in[0,1]$,
\begin{align*}
	V_{\alpha,p}(t,Y_{n,t})= & \int_0^t \Bigl(-\alpha V_{\alpha,p}(s,Y_{n,s}) + p Y_{n,s} V_{\alpha,p-2}(s,Y_{n,s})\cdot \bigl(G'(\eul_{n,s})\cdot  \mu_n(\eul_{n,\usn}) - \widetilde \mu_n(\eultr_{n,\usn}) \\& \qquad\qquad \qquad +\tfrac{1}{2}G''(\eul_{n,s})\cdot\sigma_n^2(\eul_{n,\usn})\bigr) \Bigr)\, ds \nonumber\\
	& + \frac{p(p-1)}{2} \int_0^t V_{\alpha,p-2}(s,Y_{n,s})\bigl( G'(\eul_{n,s})\cdot\sigma_n(\eul_{n,\usn}) -\widetilde \sigma_n(\eultr_{n,\usn})  \bigr)^2\, ds + M_{n,\alpha,t},
\end{align*}
where
\[
M_{n,\alpha,t}= p\int_0^t  Y_{n,s} V_{\alpha,p-2}(s,Y_{n,s})\cdot\bigl( G'(\eul_{n,s})\cdot\sigma_n(\eul_{n,\usn}) -\widetilde \sigma_n(\eultr_{n,\usn})  \bigr) \,dW_s.
\]
Using the fact that $G$ is Lipschitz continuous and  $G'$ is bounded as well as
Lemma~\ref{tamedcoeff1}(i)  and \eqref{tmg2} we obtain that there exist $c_1,c_2\in (0,\infty)$ such that, for all $\alpha\in(0, \infty)$,
all $n\in\N$ and all $s\in[0,1]$,
\begin{align*}
	&	 Y^2_{n,s} V^2_{\alpha, p-2}(s,Y_{n,s})\cdot |G'(\widehat X_{n,s})\cdot\sigma_n(\eul_{n,\usn}) -\widetilde \sigma_n(\eultr_{n,\usn})|^2 \\
	& \qquad  \le  c_1\cdot (1+|\widehat X_{n,s}|^{2p}  + |\eultr_{n,s})|^{2p})\cdot\bigl(1+|\eul_{n,\usn}|^{2\ell_\sigma + 2} +  |\eultr_{n,\usn}|^{2\ell_\sigma +2}\bigr) \\
	& \qquad  \le  c_2\cdot (1+\sup_{t\in[0,1]}|\widehat X_{n,t}|^{2p+2\ell_\sigma + 2}  + \sup_{t\in[0,1]}|\eultr_{n,t}|^{2p+2\ell_\sigma + 2}).
\end{align*}
Employing \eqref{bbb2} we therefore conclude that there exists $c\in (0,\infty)$ such that, for all  $\alpha\in(0, \infty)$ and all $n\in\N$,
\begin{align*}
&	\EE\bigl[\sup_{s\in[0,1]} Y^2_{n,s} V^2_{\alpha, p-2}(s,Y_{n,s})\cdot |G'(\widehat X_{n,s})\cdot\sigma_n(\eul_{n,\usn}) -\widetilde \sigma_n(\eultr_{n,\usn})|^2\bigr] \\
& \qquad \qquad \le  c\cdot	\bigl( 1 + \EE\bigl[\sup_{s\in[0,1]} |\widehat X_{n,s}|^{2p+2\ell_\sigma + 2}\bigr]  + \EE\bigl[\sup_{s\in[0,1]} |\eultr_{n,s})|^{2p+2\ell_\sigma + 2}\bigr]\bigr) < \infty.
\end{align*}
Hence, for  all $\alpha\in(0, \infty)$ and all $n\in\N$,   the stochastic process $(M_{\alpha,n,t})_{t\in[0,1]}$ is
 a martingale. Thus, for all $\alpha\in(0, \infty)$, all $n\in\N$ and all stopping times $\tau$ with $\tau\le 1$,
\begin{equation}\label{stopp1}
	\begin{aligned}
		& \EE\bigl[ V_{\alpha,p}(\tau,Y_{n,\tau})\bigr] \\
		&\qquad = \EE\Bigl[ \int_0^\tau \Bigl( p Y_{n,s} V_{\alpha,p-2}(s,Y_{n,s})\cdot \bigl(\widetilde \mu(G(\widehat X_{n,s})) - \widetilde \mu(\eultr_{n,s})+ A_{1,n,s}+A_{2,n,s}\bigr) \\
				& \qquad\quad  -\alpha V_{\alpha,p}(s,Y_{n,s}) + \frac{p(p-1)}{2}   V_{\alpha,p-2}(s,Y_{n,s})\bigl( \widetilde \sigma(G(\widehat X_{n,s})) - \widetilde \sigma(\eultr_{n,s})+ B_{n,s}\bigr)^2\, ds \Bigr) \Bigr],
	\end{aligned}
\end{equation}
where
	\begin{align*}
A_{1,n,s} & = G'(\eul_{n,s})\cdot  \mu_n(\eul_{n,\usn}) - \widetilde \mu_n(\eultr_{n,\usn})+\tfrac{1}{2}G''(\eul_{n,\usn})\cdot\sigma_n^2(\eul_{n,\usn}) - \bigl(\widetilde \mu(G(\widehat X_{n,s})) - \widetilde \mu(\eultr_{n,s})\bigr),\\
A_{2,n,s} & = \tfrac{1}{2}\sigma_n^2(\eul_{n,\usn})\cdot (G''(\eul_{n,s}) - G''(\eul_{n,\usn})),\\
B_{n,s}	& =  G'(\eul_{n,s})\cdot\sigma_n(\eul_{n,\usn}) -\widetilde \sigma_n(\eultr_{n,\usn}) - \bigl( \widetilde \sigma(G(\widehat X_{n,s})) - \widetilde \sigma(\eultr_{n,s})\bigr).
	\end{align*}
Since $p_1>2$, we can choose $\beta>0$ such that $(p-1)(1+\beta)\le p_1-1$.
By \eqref{squares},
for all $n\in\N$ and all $s\in[0,1]$,
\[
\bigl( \widetilde \sigma(G(\widehat X_{n,s})) - \widetilde \sigma(\eultr_{n,s})+ B_{n,s}\bigr)^2\le (1+\beta) \bigl( \widetilde \sigma(G(\widehat X_{n,s})) - \widetilde \sigma(\eultr_{n,s})\bigr)^2 + (1+1/\beta)B_{n,s}^2.
\]
Moreover, by Young's inequality,   for all $\alpha\in(0, \infty)$,
all $n\in\N$ and all $s\in[0,1]$,
\[
 V_{\alpha,p-2}(s,Y_{n,s})\cdot (A_{1,n,s}^2 + B_{n,s}^2) \le 2 \frac{p-2}{p} V_{\alpha,p}(s,Y_{n,s}) + \frac{2}{p} \exp(-\alpha\cdot s) (|A_{1,n,s}|^p+ |B_{n,s}|^p).
\]
Using Lemma~\ref{transform1}
we therefore obtain that there exist $c_1,c_2,c_3\in (0,\infty)$ such that, for all $\alpha\in(0, \infty)$,
all $n\in\N$ and all $s\in[0,1]$,
\begin{equation}\label{stopp2}
	\begin{aligned}
& p Y_{n,s} V_{\alpha,p-2}(s,Y_{n,s})\cdot \bigl(\widetilde \mu(G(\widehat X_{n,s})) - \widetilde \mu(\eultr_{n,s})+ A_{1,n,s} + A_{2,n,s}\bigr) \\
& \qquad\qquad \qquad  -\alpha V_{\alpha,p}(s,Y_{n,s}) + \frac{p(p-1)}{2}   V_{\alpha,p-2}(s,Y_{n,s})\bigl( \widetilde \sigma(G(\widehat X_{n,s})) - \widetilde \sigma(\eultr_{n,s})+ B_{n,s}\bigr)^2	\\
& \qquad \le  \frac{p}{2}  V_{\alpha,p-2}(s,Y_{n,s})\cdot \bigl(2Y_{n,s}\cdot (\widetilde \mu(G(\widehat X_{n,s})) - \widetilde \mu(\eultr_{n,s}))\\
		& \qquad\qquad\qquad\qquad \qquad\qquad \qquad +(p_1-1) \cdot ( \widetilde \sigma(G(\widehat X_{n,s})) - \widetilde \sigma(\eultr_{n,s}))^2\bigr)\\
				& \qquad\qquad + p Y_{n,s} V_{\alpha,p-2}(s,Y_{n,s})\cdot  A_{2,n,s} + 2p V_{\alpha,p-2}(s,Y_{n,s})\cdot ( Y^2_{n,s} +  A^2_{1,n,s})\\
				& \qquad\qquad\qquad -\alpha V_{\alpha,p}(s,Y_{n,s}) + \frac{p(p-1)}{2}    V_{\alpha,p-2}(s,Y_{n,s})\cdot (1+1/\beta)\cdot B_{n,s}^2\\
				& \qquad \le  \frac{p}{2}  V_{\alpha,p-2}(s,Y_{n,s})\cdot c_1 Y_{n,s}^2+p V_{\alpha,p-1}(s,Y_{n,s})\cdot  |A_{2,n,s}| + 2p V_{\alpha,p-2}(s,Y_{n,s})\cdot Y_{n,s}^2 	\\
					& \qquad\qquad -\alpha V_{\alpha,p}(s,Y_{n,s})	+ c_2  \cdot V_{\alpha,p-2}(s,Y_{n,s})\cdot (A_{1,n,s}^2 + B_{n,s}^2)\\
							& \qquad \le   V_{\alpha,p}(s,Y_{n,s})\cdot\Bigl( -\alpha + \frac{c_1 p}{2} +2p+  \frac{2c_2(p-2)}{p}\Bigr)	+p  V_{\alpha,p-1}(s,Y_{n,s})\cdot  |A_{2,n,s}| \\
							& \qquad\qquad\qquad	+ \frac{2c_2}{p} (|A_{1,n,s}|^p+|B_{n,s}|^p).
	\end{aligned}
\end{equation}

Choose $\alpha > \frac{c_1 p}{2} +2p+  \frac{2c_2(p-2)}{p}$. We conclude from~\eqref{stopp1} and~\eqref{stopp2} that there exists $c\in (0,\infty)$ such that, for all $n\in\N$ and all stopping times $\tau$ with $\tau\le 1$,
 \begin{equation}\label{stopp3}
 		 \EE\bigl[ V_{\alpha,p}(\tau,Y_{n,\tau})\bigr] \le c \cdot \EE\Bigl[ \int_0^\tau \bigl(V_{\alpha,p-1}(s,Y_{n,s})\cdot  |A_{2,n,s}|+ |A_{1,n,s}|^p+|B_{n,s}|^p  \bigr)\, ds \Bigr].
 \end{equation}
Below we show that
there exists $c\in (0,\infty)$ such that for all $n\in\N$ and all $s\in [0,1]$,
 \begin{equation}\label{ana5}
	\EE\bigl[	|A_{1,n,s}|^p +	|B_{n,s}|^p\bigr]
		 \le\frac{c}{n^{p/2}},
\end{equation}	
and there exists $c\in (0,\infty)$ such that for all $n\in\N$,
\begin{equation}\label{eft3}
 \EE\Bigl[ \Bigl|	\int_0^1 |A_{2,n,s}|\, ds\Bigr|^p\Bigr]\le  \frac{c}{n^{p/2} }.
\end{equation}

Employing~\cite[Lemma 3.2]{gk03b} as well as~\eqref{ana5}, ~\eqref{eft3} and the Young inequality we derive from~\eqref{stopp3} that for all $\gamma\in (0,1)$ there exist $c_1,c_2,c_3, c_4 \in (0,\infty)$ such that, for all $n\in\N$,
\begin{equation}\label{etf4a}
\begin{aligned}
&	\EE\bigl[\sup_{s\in[0,1]}\exp(-\alpha s\gamma) |Y_{n,s}|^{p\gamma} \bigr] \\
& \qquad    \le 	c_1 \cdot\frac{2-\gamma}{1-\gamma}\cdot
 \EE\Bigl[ \Bigl( \int_0^1 \bigl(\exp(-\alpha s) |Y_{n,s}|^{p-1}\cdot  |A_{2,n,s}|+ |A_{1,n,s}|^p+|B_{n,s}|^p\bigr) \, ds \Bigr)^\gamma \Bigr]\\
 & \qquad   \le
 \EE\Bigl[ \bigl( \sup_{s\in[0,1]}\exp(-\alpha s\gamma) |Y_{n,s}|^{(p-1)\gamma}\bigr)\cdot  \Bigl( c_2\cdot \int_0^1   |A_{2,n,s}|\, ds \Bigr)^\gamma\Bigr]\\
 & \qquad\qquad\qquad+c_2\cdot \EE\Bigl[\Bigl( \int_0^1   ( |A_{1,n,s}|^p+|B_{n,s}|^p )\, ds \Bigr)^\gamma \Bigr]\\
  & \qquad  \le
 \EE\Bigl[ \bigl( \sup_{s\in[0,1]}\exp(-\alpha s\gamma (p-1)/p) |Y_{n,s}|^{(p-1)\gamma}\bigr)\cdot \Bigl( c_2\cdot\int_0^1   |A_{2,n,s}|\, ds \Bigr)^\gamma\Bigr] +  \frac{c_3}{n^{p\gamma/2}}\\
  & \qquad   \le
\frac{p-1}{p}\cdot \EE\Bigl[  \sup_{s\in[0,1]}\exp(-\alpha s\gamma) |Y_{n,s}|^{p\gamma}\Bigr]  + \frac{1}{p}\cdot \EE\Bigl[\Bigl( c_2\cdot\int_0^1   |A_{2,n,s}|\, ds \Bigr)^{\gamma p}\Bigr] + \frac{c_3 }{n^{p\gamma/2}}\\
 & \qquad   \le
\frac{p-1}{p}\cdot \EE\Bigl[  \sup_{s\in[0,1]}\exp(-\alpha s\gamma) |Y_{n,s}|^{p\gamma}\Bigr]  + \frac{c_4 }{n^{p\gamma/2}}.	
\end{aligned}
\end{equation}
Note that by~\eqref{bbb2} and the Lipschitz continuity of $G$, for all $\gamma\in(0,1)$ and all $n\in\N$,
\[
\EE\bigl[\sup_{s\in[0,1]}\exp(-\alpha s\gamma) |Y_{n,s}|^{p\gamma} \bigr] \leq \EE\bigl[\sup_{s\in[0,1]} |Y_{n,s}|^{p\gamma} \bigr]<\infty.
\]
Consequently, we obtain from~\eqref{etf4a} that for all $\gamma \in(0,1)$ there exists $c\in (0,\infty)$ such that, for all $n\in\N$,
\[
\EE[\sup_{s\in[0,1]} |Y_{n,s}|^{p\gamma} ]  \le \exp(\alpha\gamma)\cdot  \EE\Bigl[  \sup_{s\in[0,1]}\exp(-\alpha s\gamma) |Y_{n,s}|^{p\gamma}\Bigr] \le   \frac{c}{n^{p\gamma/2}},
\]
which implies the statement of Theorem \ref{comp}.

It remains to prove  the estimates \eqref{ana5} and \eqref{eft3}.
We first prove \eqref{ana5}.
Using the definition \eqref{tildecoeff} of $\widetilde \mu$ and $\widetilde\sigma$ we obtain that for all $n\in\N$ and all $s\in [0,1]$,
\begin{align*}
	\begin{aligned}
	A_{1,n,s} & = \bigl(G'(\eul_{n,s})- G'(\eul_{n,\usn})\cdot  \mu_n(\eul_{n,\usn}) + (G'\cdot  \mu_n+\tfrac{1}{2}G''\cdot\sigma_n^2)(\eul_{n,\usn}) - \widetilde \mu(G(\widehat X_{n,s})) \\
			& \qquad +(	\widetilde \mu(\eultr_{n,s}) -\widetilde \mu(\eultr_{n,\usn}) ) + (	\widetilde \mu(\eultr_{n,\usn}) - \widetilde \mu_n(\eultr_{n,\usn}))\\
		& = \bigl(G'(\eul_{n,s})- G'(\eul_{n,\usn})\cdot  \mu_n(\eul_{n,\usn}) +  (\widetilde \mu(G(\widehat X_{n,\usn})) - \widetilde \mu(G(\widehat X_{n,s})))\\
		& \qquad + G'(\widehat X_{n,\usn})\cdot(  \mu_n   - \mu)(\widehat X_{n,\usn}) +\tfrac{1}{2}G''(\widehat X_{n,\usn})\cdot     (\sigma_n^2- \sigma^2)(\eul_{n,\usn})  \\
		& \qquad +(	\widetilde \mu(\eultr_{n,s}) -\widetilde \mu(\eultr_{n,\usn}) ) + (	\widetilde \mu- \widetilde \mu_n)(\eultr_{n,\usn})
	\end{aligned}
\end{align*}
and
\begin{align*}
	\begin{aligned}
		B_{n,s}
		& = \bigl(G'(\eul_{n,s})- G'(\eul_{n,\usn})\cdot  \sigma_n(\eul_{n,\usn}) +  (\widetilde \sigma(G(\widehat X_{n,\usn})) - \widetilde \sigma(G(\widehat X_{n,s})))\\
		& \qquad + G'(\widehat X_{n,\usn})\cdot(  \sigma_n   - \sigma)(\widehat X_{n,\usn})  \\
		& \qquad +(	\widetilde \sigma(\eultr_{n,s}) -\widetilde \sigma(\eultr_{n,\usn}) ) + (	\widetilde \sigma- \widetilde \sigma_n)(\eultr_{n,\usn})).	
	\end{aligned}
\end{align*}
Employing the Lipschitz continuity of $G$ and $G'$, the boundedness of $G'$
and $G''$,
the fact that $\widetilde\mu$ satisfies (A2') and $\widetilde\sigma$  satisfies (A3),
\eqref{tmg2}, \eqref{tmg3} and Lemma~\ref{tamedcoeff1}(i),(v),(vi) we thus conclude that there exist $c_1,c_2,c_3,c_4\in (0,\infty)$ such that, for all $n\in\N$ and all $s\in [0,1]$,
\begin{equation}\label{ana3}
	\begin{aligned}
		|A_{1,n,s}| & \le c_1 \cdot \Bigl(|\eul_{n,s}- \eul_{n,\usn}| \cdot (1+ |\eul_{n,\usn}|^{\ell_\mu+1})\\
		&\qquad\qquad + (1+ |G(\eul_{n,\usn})|^{\ell_\mu}  + |G(\eul_{n,s})|^{\ell_\mu} )\cdot   |G(\eul_{n,s})- G(\eul_{n,\usn})| \\
				& \qquad\qquad + \frac{1}{\sqrt{n}}\cdot( 1+ |\widehat X_{n,\usn}|^{2\ell_\mu+1})  + \frac{1}{\sqrt{n}} \cdot( 1+ |\widehat X_{n,\usn}|^{\ell_\mu+2\ell_\sigma+2}) \\
					& \qquad\qquad +  (1+ |\eultr_{n,\usn}|^{\ell_\mu}  + |\eultr_{n,s}|^{\ell_\mu} )\cdot   |\eultr_{n,s}- \eultr_{n,\usn}|
		+ \frac{1}{\sqrt{n}} \cdot( 1+ |\widehat Z_{n,\usn}|^{2\ell_\mu+1})\Bigr)\\
			& \le c_2\cdot \Bigl((1+ |\eul_{n,\usn}|^{\ell_\mu+1}+ |\eul_{n,s}|^{\ell_\mu+1}   ) \cdot |\eul_{n,s}- \eul_{n,\usn}|  \\
		& \qquad\qquad +	\frac{1}{\sqrt{n}} \cdot( 1+ |\widehat X_{n,\usn}|^{\kappa } + |\widehat Z_{n,\usn}|^{\kappa} ) \\
		& \qquad\qquad +  (1+ |\eultr_{n,\usn}|^{\ell_\mu}  + |\eultr_{n,s}|^{\ell_\mu} )\cdot   |\eultr_{n,s}- \eultr_{n,\usn}| \Bigr),
	\end{aligned}
\end{equation}
and, similarly,
\begin{equation}\label{ana4}
	\begin{aligned}
		|B_{n,s}| & \le c_3 \cdot\Bigl(|\eul_{n,s}- \eul_{n,\usn}| \cdot (1+ |\eul_{n,\usn}|^{\ell_\sigma+1})\\
		&\qquad\qquad + (1+ |G(\eul_{n,\usn})|^{\ell_\sigma}  + |G(\eul_{n,s})|^{\ell_\sigma} )\cdot   |G(\eul_{n,s})- G(\eul_{n,\usn})| \\
		& \qquad\qquad + \frac{1}{\sqrt{n}} \cdot( 1+ |\widehat X_{n,\usn}|^{\ell_\mu + \ell_\sigma+1}) +  (1+ |\eultr_{n,\usn}|^{\ell_\sigma}  + |\eultr_{n,s}|^{\ell_\sigma} )\cdot   |\eultr_{n,s}- \eultr_{n,\usn}|\\
			& \qquad\qquad
		+ \frac{1}{\sqrt{n}} \cdot( 1+ |\widehat Z_{n,\usn}|^{\ell_\mu+\ell_\sigma + 1})\Bigr)\\
		& \le c_4\cdot\Bigl((1+ |\eul_{n,\usn}|^{\ell_\sigma+1}+ |\eul_{n,s}|^{\ell_\sigma+1}   ) \cdot |\eul_{n,s}- \eul_{n,\usn}|  \\
		& \qquad\qquad +	\frac{1}{\sqrt{n}} \cdot( 1+ |\widehat X_{n,\usn}|^{\kappa} + |\widehat Z_{n,\usn}|^{\kappa} ) \\
		& \qquad\qquad +  (1+ |\eultr_{n,\usn}|^{\ell_\sigma}  + |\eultr_{n,s}|^{\ell_\sigma} )\cdot   |\eultr_{n,s}- \eultr_{n,\usn}|\Bigr).
	\end{aligned}
\end{equation}
 Combining~\eqref{ana3} with~\eqref{ana4} and observing that $\ell_\sigma \le \ell_\mu/2$ yields  that there exists $c\in (0,\infty)$ such that for all $n\in\N$ and all $s\in [0,1]$,
 \begin{align*}
 	\begin{aligned}
	|A_{1,n,s}|^p +	|B_{n,s}|^p  & \le
		c\cdot\Bigl((1+ |\eul_{n,\usn}|^{\ell_\mu+1}+ |\eul_{n,s}|^{\ell_\mu+1}   )^p \cdot |\eul_{n,s}- \eul_{n,\usn}|^p \\
			& \qquad\qquad +	\frac{1}{n^{p/2}} \cdot( 1+ |\widehat X_{n,\usn}|^{\kappa} + |\widehat Z_{n,\usn}|^{\kappa} )^p \\
				& \qquad\qquad +  (1+ |\eultr_{n,\usn})|^{\ell_\mu}  + |\eultr_{n,s})|^{\ell_\mu} )^p\cdot   |\eultr_{n,s}- \eultr_{n,\usn}| ^p\Bigr).
 	\end{aligned}
\end{align*}		
Hence, by the fact that
\[
p(\ell_\mu +\ell_\sigma +1) < p(\ell_\mu +\ell_\sigma +2)\le p\kappa \le p_0
\]
and the
estimates~\eqref{tmg002a} and~\eqref{tmg003} in Lemma~\ref{eulprop} we conclude that there exist $c_1,c_2\in (0,\infty)$ such that, for all $n\in\N$ and all $s\in [0,1]$,
\begin{align*}
&	\EE\bigl[	|A_{1,n,s}|^p +	|B_{n,s}|^p\bigr] \\
		&\quad \le c_1 \cdot\Bigl( \bigl(1+\sup_{t\in[0,1]}\EE\bigl[  |\eul_{n,t}|^{p(\ell_\mu+\ell_\sigma+2)}  \bigr]^{\frac{\ell_\mu+1}{\ell_\mu+\ell_\sigma+2}}\bigr)\cdot \sup_{t\in[0,1]}\EE\bigl[  |\eul_{n,t}- \eul_{n,\utn}|^{p\frac{\ell_\mu+\ell_\sigma+2}{\ell_\sigma+1}}\bigr]^{\frac{\ell_\sigma+1}{\ell_\sigma+\ell_\sigma+2}}\\
		& \qquad\quad +	\frac{1}{n^{p/2}} \cdot \bigl(1+\sup_{t\in[0,1]}\EE\bigl[  |\eul_{n,t}|^{p\kappa}\bigr] + \sup_{t\in[0,1]}\EE\bigl[  |\eultr_{n,t}|^{p\kappa}\bigr] \bigr)\\
			&\qquad \quad +  \bigl(1+\sup_{t\in[0,1]}\EE\bigl[  |\eultr_{n,t}|^{p(\ell_\mu+\ell_\sigma+1)}  \bigr]^{\frac{\ell_\mu}{\ell_\mu+\ell_\sigma+1}}\bigr)\cdot \sup_{t\in[0,1]}\EE\bigl[  |\eultr_{n,t}- \eultr_{n,\utn}|^{p\frac{\ell_\mu+\ell_\sigma+1}{\ell_\sigma+1}}\bigr]^{\frac{\ell_\sigma+1}{\ell_\sigma+\ell_\sigma+1}}\Bigr)\\
			&\quad \le\frac{c_2}{n^{p/2}},
\end{align*}	
which finishes the proof of \eqref{ana5}.

Next, we prove \eqref{eft3}.
Put
\[
B= \Bigl(\bigcup_{i=1}^{k+1} (\xi_{i-1}, \xi_{i})^2\Bigr)^c
\]
and note that
$B=\bigcup_{i=1}^k \{(x,y)\in\R^2:  (x-\xi_{i})\cdot(y-\xi_{i})\leq 0\}$.
Using Lemma~\ref{tamedcoeff1}(i), ~\eqref{tmg2} and Lemma~\ref{lemx1}(iv) we
obtain that there exists
$ c\in(0, \infty)$ such that, for all $n\in\N$ and all  $x,y\in\R$,
\begin{align*}
	&|\sigma_n^2(y)\cdot (G''(x)-G''(y))|\leq \begin{cases}
		c\cdot (1+|y|^{2\ell_\sigma+2})\cdot|x-y|, &\text{ if } (x,y)\in B^c, \\
		c\cdot (1+|y|^{2\ell_\sigma+2}),&\text{ if } (x,y)\in B.
	\end{cases}
\end{align*}
Hence there exists
$ c\in(0, \infty)$ such that, for all $n\in\N$,
\begin{equation}\label{end1}
	\begin{aligned}
	\int_0^1 |A_{2,n,s}|\, ds
		& \leq c\cdot \Bigl(\int_0^1(1+|\eul_{n,\usn}|^{2\ell_\sigma+2})\cdot |\eul_{n,s}-\eul_{n,\usn} |\, ds \\
		& \qquad\qquad +   \int_0^1(1+|\eul_{n,\usn}|^{2\ell_\sigma+2})\cdot 1_{\{(\eul_{n,s}, \eul_{n,\usn})\in B\}}\, ds\Bigr)\\
			& \leq c\cdot \Bigl(\int_0^1(1+|\eul_{n,\usn}|^{2\ell_\sigma+2})\cdot |\eul_{n,s}-\eul_{n,\usn} |\, ds \\
		& \qquad\qquad +  \bigl(1+\sup_{s\in [0,1]}|\eul_{n,s}|^{2\ell_\sigma+2}\bigr)\cdot  \int_0^11_{\{(\eul_{n,s}, \eul_{n,\usn})\in B\}}\, ds\Bigr).
	\end{aligned}
\end{equation}
Observing that $2\ell_\sigma\le \ell_\mu$ and
\[
 p(\ell_\mu + \ell_\sigma +3) \le p\kappa \le p_0
\]
 and using
the estimates ~\eqref{tmg002a} and~\eqref{tmg003} in Lemma \ref{eulprop}
we obtain that there exist $c_1,c_2, c_3, c_4\in (0,\infty)$ such that, for all $n\in\N$ and all $s\in[0,1]$,
\begin{align*}
 & \EE \bigl[\bigl((1+|\eul_{n,\usn}|^{2\ell_\sigma+2})\cdot |\eul_{n,s}-\eul_{n,\usn} |	\bigr)^p\bigr]\\	&\qquad  \le c_1\cdot \EE \bigl[(1+|\eul_{n,\usn}|^{p(2\ell_\sigma+2)})\cdot |\eul_{n,s}-\eul_{n,\usn} |^p\bigr] \\
 &\qquad  \le c_2\cdot \EE\bigl[\bigl(1+|\eul_{n,\usn}|^{p(2\ell_\sigma+2)\frac{\ell_\mu+\ell_\sigma+3}{\ell_\mu+\ell_\sigma+2}}\bigr)
 	\bigr]^{\frac{\ell_\mu+\ell_\sigma+2}{\ell_\mu+\ell_\sigma+3}}\cdot
 	 \EE \bigl[ |\eul_{n,s}-\eul_{n,\usn} |^{p\frac{\ell_\mu+\ell_\sigma+3}{\ell_\sigma+1}}\bigr]^{\frac{\ell_\sigma+1}{\ell_\mu+\ell_\sigma+3}}\\
 &\qquad  \le c_3\cdot \Bigl(1+\sup_{t\in[0,1]}\EE\bigl[|\eul_{n,\usn}|^{p_0}\bigr]^{\frac{\ell_\mu+\ell_\sigma+2}{\ell_\mu+\ell_\sigma+3}}\Bigr)\cdot \frac{1}{n^{p/2} } \le  \frac{c_4}{n^{p/2} },
\end{align*}
which yields the existence of $c\in (0,\infty)$ such that, for all $n\in\N$,
\begin{equation}\label{eft1}
	\EE\Bigl[ \Bigl|\int_0^1(1+|\eul_{n,\usn}|^{2\ell_\sigma+2})\cdot |\eul_{n,s}-\eul_{n,\usn} |\, ds\Bigr|^p\Bigr] \le \frac{c}{n^{p/2} }.
	\end{equation}
Observing that
\[
p(\ell_\mu + 2\ell_\sigma +2) < p_0
\]
and using the estimate ~\eqref{tmg003a} in Lemma~\ref{eulprop}  and Proposition~\ref{prop1}   we derive that there exist $c_1,c_2\in (0,\infty)$ such that, for all $n\in\N$,
\begin{equation}\label{eft2}
	\begin{aligned}
& \EE\Bigl[ \Bigl|(1+\sup_{s\in [0,1]}|\eul_{n,s}|^{2\ell_\sigma+2})\cdot  \int_0^1 1_{\{(\eul_{n,s}, \eul_{n,\usn})\in B\}}\, ds\Bigr|^p\Bigr]\\
& \qquad \le \Bigl(1+\EE\bigl[\sup_{s\in [0,1]}|\eul_{n,s}|^{p(\ell_\mu +2\ell_\sigma+2}\bigr]^\frac{2\ell_\sigma+2}{\ell_\mu+2\ell_\sigma+2}\Bigr) \\
& \qquad\qquad\qquad\cdot \EE\Bigl[\Bigr| \int_0^1 1_{\{(\eul_{n,s}, \eul_{n,\usn})\in B\}}\, ds\Bigr|^{p\frac{\ell_\mu+2\ell_\sigma+2}{\ell_\mu}}\Bigr]^{\frac{\ell_\mu}{\ell_\mu+2\ell_\sigma+2}} \\
& \qquad \le c_1\cdot \sum_{i=1}^k \EE\Bigl[\Bigl|\int_0^1\ind_{\{(\eul_{n,s}-\xi_i)\cdot (\eul_{n,\usn}-\xi_i)\leq 0\}}\, ds\Bigr|^{p\frac{\ell_\mu+2\ell_\sigma+2}{\ell_\mu}}\Bigr]^{\frac{\ell_\mu}{\ell_\mu+2\ell_\sigma+2}} \le  \frac{c_2}{n^{p/2} }.
	\end{aligned}
\end{equation}
Combining~\eqref{end1},~\eqref{eft1} and~\eqref{eft2} we conclude that there exists $c\in (0,\infty)$ such that, for all $n\in\N$,
\[
 \EE\Bigl[ \Bigl|	\int_0^1 |A_{2,n,s}|\, ds\Bigr|^p\Bigr]\le  \frac{c}{n^{p/2} },
\]
which finishes the proof of \eqref{eft3} and completes the proof of Theorem~\ref{comp}.

\subsection{Proof of Theorem~\ref{Thm2}}\label{4.5}
Clearly, for all $p\in [1,\infty)$, all $q\in [1,\infty]$ and all $n\in\N$,
\begin{equation}\label{two0}
\begin{aligned}
\EE\bigl[\|X-\overline X_n\|_q^p\bigr]^{1/p}
& \le \EE\bigl[\|X-\eul_n\|_q^p\bigr]^{1/p} +\EE\bigl[\|\eul_n-\overline X_n\|_q^p\bigr]^{1/p}\\
& \le \EE\bigl[\|X-\eul_n\|_\infty^p\bigr]^{1/p} +\EE\bigl[\|\eul_n-\overline X_n\|_q^p\bigr]^{1/p}.
\end{aligned}
\end{equation}

For
$n\in\N$
define a stochastic process
$\overline W_n = (\overline W_{n,t})_{t\in[0,1]}$ by
\[
\overline W_{n,t} = (n\cdot t-i)\cdot W_{n,(i+1)/n} + (i+1-n\cdot t)\cdot  W_{n,i/n}
\]
for $t\in [i/n,(i+1)/n]$
and $i\in\{0,\ldots,n-1\}$.
Then for all  $q\in [1,\infty]$ and all $r\in [1,\infty)$
there exists $c\in (0,\infty)$ such that, for all $n\in\N$,
\begin{equation}\label{two01}
	\bigl(\EE\bigl[\|W-\overline W_n\|_q^r\bigr]\bigr)^{1/r}\leq \begin{cases}
		c/\sqrt n, & \text{ if }q < \infty,\\ c\sqrt{\ln (n+1)}/\sqrt{n}, & \text{ if }q = \infty,
	\end{cases}
\end{equation}
see, e.g.~\cite{Speckman79} for the case $q\in[1, \infty)$ and~\cite{Faure92} for the case $q=\infty$.

Note that for all $n\in\N$ and all $t\in[0,1]$,
\begin{align*}
	|\eul_{n,t} - \overline X_{n,t}| & =
	\Bigl|\sum_{i=0}^{n-1}\sigma_n(\eul_{n,i/n})\cdot \ind_{[i/n,(i+1)/n]}(t)\cdot (W_t- \overline W_{n,t})\Bigr|\\
	& \le \sup_{s\in[0,1]} |\sigma_n(\eul_{n,s})|\cdot |W_t-\overline W_{n,t}|.
\end{align*}
Hence, by Lemma~\ref{tamedcoeff1}(i), ~\eqref{tmg2} and the estimate \eqref{tmg003a} in Lemma~\ref{eulprop},  for all $p\in [1,p_0/(\ell_\sigma+2))$ and all $q\in [1,\infty]$  there exist $c_1,c_2,c_3\in (0,\infty)$ such that, for all $n\in\N$,
\begin{equation}\label{eft7}
\begin{aligned}
	&\bigl(\EE\bigl[\|\eul_n-\overline X_n\|_q^p\bigr]\bigr)^{1/p}\\
	 & \qquad\le c_1\cdot \EE\bigl[\bigl(1+\sup_{s\in[0,1]}|\eul_{n,s}|^{p(\ell_\sigma +1)}\bigr)\cdot \|W-\overline W_n\|_q^{p}\bigr]^{1/{p}}\\
		&\qquad\le c_2 \cdot \bigl(1 + \EE\bigl[\sup_{s\in[0,1]}|\eul_{n,s}|^{p(\ell_\sigma +2)}\bigr]^{\frac{\ell_\sigma+1}{p(\ell_\sigma+2)}}\bigr)\cdot \EE\bigl[ \|W-\overline W_n\|_q^{p(\ell_\sigma+2)} \bigr]^{\frac{1}{p(\ell_\sigma+2)}}\\
	&\qquad\le c_3\cdot \EE\bigl[ \|W-\overline W_n\|_q^{p(\ell_\sigma+2)} \bigr]^{\frac{1}{p(\ell_\sigma+2)}}.
\end{aligned}
\end{equation}
Put $\kappa = \ell_\mu+\max(\ell_\mu,2\ell_\sigma + 2)+1$. Since $p_0/(\ell_\sigma+2) > p_0/\kappa$ we may  combine~\eqref{two0} to~\eqref{eft7} with Theorem~\ref{Thm1} to obtain Theorem~\ref{Thm2}.

\bibliographystyle{acm}
\bibliography{bibfile}

\end{document}